\documentclass{amsart}

\usepackage{mathtools, amssymb, amsfonts, amsthm}

\usepackage[colorlinks=true, linkcolor=red, citecolor=teal, urlcolor=teal]{hyperref}
\usepackage{cleveref}
\usepackage{graphicx} 
\usepackage{booktabs} 
\usepackage{fullpage}
\usepackage{xcolor}
\usepackage{accents}
\usepackage{float}
\usepackage{mathrsfs}

\usepackage{thmtools, thm-restate}

\definecolor{Teal}{rgb}{0,0.784,0.784} 

\usepackage{etoolbox}

\newcommand{\SD}{\mathrm{SD}}
\newcommand{\vv}{\mathbf{v}}

\usepackage{etoolbox}

\newcommand{\includesymbol}[1]{\ensuremath{%
	\mathchoice
		{\raisebox{-.1mm}{\includegraphics[height=2ex]{#1}}}
		{\raisebox{-.1mm}{\includegraphics[height=2ex]{#1}}}
		{\raisebox{-.6mm}{\includegraphics[height=2ex]{#1}}}
		{\raisebox{-.5mm}{\includegraphics[height=2ex]{#1}}}
}} 

\robustify{\includesymbol}
\newcommand{\bridge}{\includesymbol{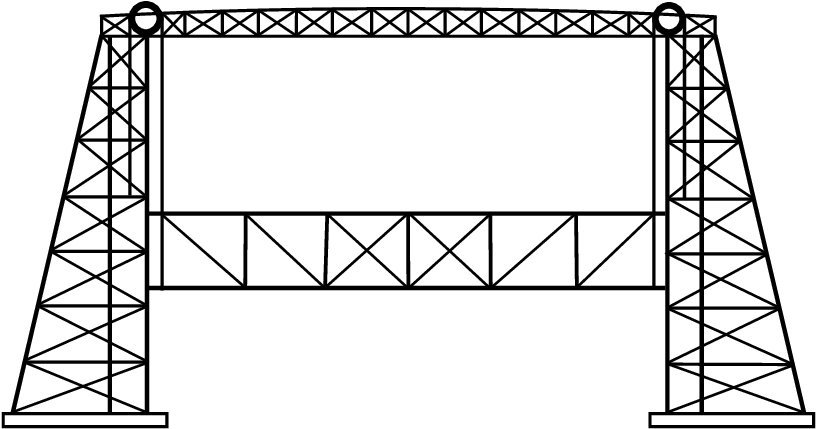}}
\newcommand{\Bridge}{{\bridge}}

\newtheorem{theorem}{Theorem}

\newtheorem{conjecture}[theorem]{Conjecture}

\newtheorem{lemma}[theorem]{Lemma}
\newtheorem{proposition}[theorem]{Proposition}

\theoremstyle{definition}
\newtheorem{definition}[theorem]{Definition}
\newtheorem{example}[theorem]{Example}

\title{Orbits of toric promotion on bridge sums}
\author{Kerry Seekamp}
\address{Department of Mathematical Sciences, Smith College, Northampton, MA 01063, \href{mailto:kseekamp@smith.edu}{kseekamp@smith.edu}}

\begin{document}


\begin{abstract}
    In 2023, Defant introduced \textit{toric promotion} as a cyclic analogue of Sch\"utzenberger's well known promotion operator. Toric promotion is defined by a choice of simple graph $G$ and acts on the labeling of $G$ by a series of involutions. Defant described the orbit length of toric promotion on trees and showed that it does not depend on the initial labeling; we prove an analogous result for complete graphs. A natural question is how toric promotion behaves under certain graph operations. In the main results of this article, we analyze the orbits of toric promotion under the \textit{bridge sum} graph operation, which joins two graphs by adding an edge between a vertex of each graph. We show that the orbit length of toric promotion on any graph constructed via a bridge sum of a tree or a complete graph with a simple graph does not depend on the restriction of the initial labeling to the tree or complete subgraph. Additionally, we describe the orbit lengths of toric promotion on the bridge sums of two complete graphs and the bridge sums of a tree with a complete graph, and show that they do not depend on the initial labeling. Finally, we describe the orbit length of toric promotion on the corona product of a complete graph with any tree, and show that it does not depend on the initial labeling.
    
\end{abstract}

\maketitle
\section{Introduction}

{\color{blue}Toric promotion} was first introduced by Defant \cite{D2023} as a cyclic analogue of Sch\"utzenberger's \cite{Schutzenberger63, Schutzenberger72} well studied promotion operator. Sch\"utzenberger's promotion is defined by a choice of finite poset and acts on the linear extensions of that poset; toric promotion is defined by a choice of {\color{blue}simple graph}, an undirected, unweighted graph with no self-loops and no multiple edges,  and acts on the labeling of that graph. Specifically, toric promotion permutes the labels of the chosen graph by a series of involutions that depend on adjacency within the chosen graph. 

Much has is known about which posets behave nicely under Sch\"utzenberger's promotion \cite{stanley2008,HopkinsRubey2022}. Moreover, the orbits of Sch\"utzenberger's promotion have been extensively studied in the context of Kuperberg's $\mathfrak{sl}_3$ webs \cite{Kuperberg96}, where promotion corresponds to rotation of the web graph \cite{Tymoczko2012,PPR09}. In the main results of this article, we analyze the behavior of toric promotion under the bridge sum graph operation. 

We will begin by defining toric promotion. Let $G=(V,E)$ be a simple graph on $\nu$ vertices. A {\color{blue}labeling} $\sigma$ of $G$ is a bijection $V \rightarrow \mathbb{Z}/\nu \mathbb{Z}$, where $\mathbb{Z}/ \nu \mathbb{Z} = \{1,2,\ldots,\nu\}$ and addition is done modulo $\nu$. We denote the set of labelings of $G$ by $\Lambda_G$. In this article, we write all labelings as permutations in one-line notation. Furthermore, a {\color{blue}state} of toric promotion is a pair $(\sigma, i)$, where $\sigma \in \Lambda_G$ is a labeling, and $i\in \mathbb{Z}/ \nu \mathbb{Z}$ is a label of a vertex in $V$. We denote the set of all possible states of toric promotion on $G$ by $\Omega_G$.

\begin{definition}Toric promotion is the map $\text{TPro}: \Omega_G \rightarrow \Omega_G$ such that 
\[\text{TPro}(\sigma,i)=
\begin{cases*}
(\sigma, i+1) & if $\{\sigma^{-1}(i),\sigma^{-1}(i+1)\} \in E;$\\
((i,i+1)\circ \sigma,i+1) & if  $\{\sigma^{-1}(i),\sigma^{-1}(i+1)\}\not \in E,$
\end{cases*}\]

where $(i, i+1)$ is the involution that swaps the labels $i$ and $i+1$.
\end{definition}

In other words, to perform toric promotion on a simple graph $G$, we look at the labels $i$ and $i+1$. If $\sigma$ assigns $i$ and $i+1$ to vertices of $G$ that are
    \begin{itemize}
        \item adjacent, then we update $i$ to $i+1$, and leave the labeling $\sigma$ unchanged;
        \item  otherwise, we update $i$ to $i+1$, and update $\sigma$ by swapping the labels $i,i+1$.
    \end{itemize}
    
The {\color{blue}length} of the orbit of toric promotion containing the initial state $(\sigma_0,i_0)$ is the minimum number of iterations of toric promotion that are required to get from $(\sigma_0,i_0)$ back to $(\sigma_0,i_0)$. Notice that $\text{TPro}$ is a permutation, and thus bijective. Hence, every element belongs to some orbit. We give an example of an orbit of toric promotion in Figure~\ref{EX: toric promotion orbit}.

\begin{figure}[H]
\begin{center}
    \includegraphics[width=.9\linewidth]{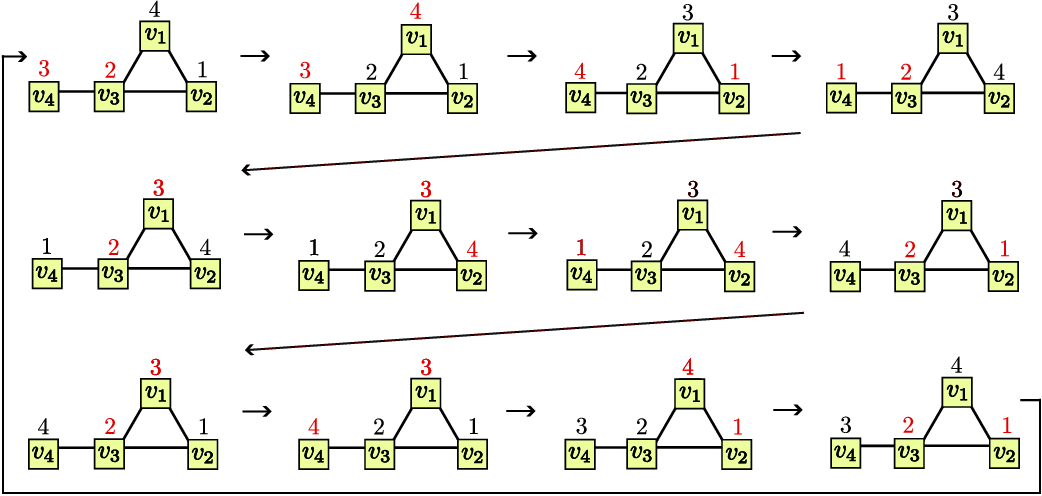}
\end{center}
\vspace{.3cm}
\caption{Let $G$ be the graph on $4$ vertices shown in yellow. We illustrate the orbit of toric promotion on $G$ containing the initial state $(4123,2)$. Each copy of $G$ shows a state of this orbit with the labels $i$ and $i+1$ shown in red.}
\label{EX: toric promotion orbit}
\end{figure}

Defant \cite{D2023} describes the orbit length of toric promotion on any tree graph as follows.

\begin{theorem} [\cite{D2023}] \label{THM: Tree}
Let $T$ be any tree on $m$ vertices. Then all orbits of toric promotion on $T$ have length $m(m-1)$. 
\end{theorem}

Notice that this expression does not depend on the initial labeling. In our first result, we give the orbit length of toric promotion on any complete graph and show that it is independent of the initial labeling. For convenience, throughout this article, $m$ will always be the number of vertices in a tree, and $n$ will always be the number of vertices in a complete graph. 

\begin{proposition}\label{THM:CompleteGraphs}
Let $K$ be the complete graph on $n$ vertices. Then all orbits of toric promotion on $K$ have length $n$.
\end{proposition} 

Since toric promotion is defined by a choice of graph, a natural question is how the orbits of toric promotion behave under certain graph operations. In the main results of this paper, we describe the behavior of toric promotion under the {\color{blue}bridge sum} graph operation, which joins two graphs by adding an edge between a vertex of each graph. The added edge is a bridge in the resulting graph, giving rise to the name bridge sum.

\begin{definition}
\label{DEF: bridge sum}
    Let $G_1=(V_1,E_1)$ and $G_2=(V_2,E_2)$, such that $v_i\in V_1$ and $v_j\in V_2$. Then the \textit{bridge sum} $G_1(v_i)\Bridge G_2(v_j)$ is the graph obtained by adding the edge $\{v_i,v_j\}$. In other words, the bridge sum $G_1(v_i)\Bridge G_2(v_j)$ is the graph with vertex set $V'=V_1 \cup V_2$ and edge set $E'=E_1\cup E_2\cup\{v_i,v_j\}$. 
\end{definition}
We show an example of a bridge sum in \Cref{Fig:BridgeSum}.

\begin{figure}[H]
\begin{center}
 \includegraphics[width=\linewidth]{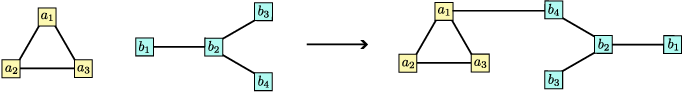}
\end{center}
\caption{Let $G_1$ be the cycle graph on $3$ vertices and let $G_2$ be a tree on $4$ vertices. On the left, we show $G_1$ (yellow) and $G_2$ (blue). On the right, we show $G_1(a_1)\Bridge G_2(b_4)$.}
\label{Fig:BridgeSum}
\end{figure}

In the following results, we consider graphs constructed by a bridge sum (at any two vertices) of a simple graph and either a tree or a complete graph. We show that the orbit length of toric promotion on graphs of this form does not depend on the restriction of the initial labeling to the tree or complete subgraph.

\begin{theorem}
Fix $m$ and a simple graph $S$. Let $G$ be a bridge sum of $S$ with an $m$-vertex tree $T$, and let $\sigma$ be any labeling of $G$. Then the length of the orbit of toric promotion containing the state $(\sigma,i)$ depends only on $m$, $S$, and $\phi$, the restriction of $\sigma$ to $S$ (in particular, it does not depend on the specific choice of $T$).
\label{THM:SimpleBridgeTree}
\end{theorem}

\begin{theorem}
Fix $n$ and a simple graph $S$. Let $G$ be a bridge sum of $S$ with the complete graph on $n$ vertices $K$, and let $\sigma$ be any labeling of $G$. Then the length of the orbit containing the state $(\sigma,i)$ depends only on $n$, $S$, and $\phi$, the restriction of $\sigma$ to $S$.
\label{THM:SimpleBridgeComplete}
\end{theorem}

Theorems~\ref{THM:SimpleBridgeTree} and \ref{THM:SimpleBridgeComplete} apply to an extensive collection of graphs and show two cases where toric promotion behaves naturally under the bridge sum operation. Since the orbit length of toric promotion on trees does not depend on the initial labeling, and the orbit length of toric promotion on complete graphs does not depend on the initial labeling, it is natural that the orbit length of toric promotion on any bridge sum of a simple graph with either a tree or a complete graph does not depend on the restriction of the initial labeling to the tree subgraph or the complete subgraph. 

In the next set of results, we give the specific orbit length of toric promotion on graphs constructed by certain bridge sums. First, we describe the orbit length of toric promotion on graphs constructed by the bridge sum of two complete graphs. Then, we describe the orbit length of toric promotion on graphs constructed by a bridge sum of a complete graph and a tree. Notice that the bridge sum of two trees is also a tree. Therefore, this case is shown by Defant in Theorem~\ref{THM: Tree} \cite{D2023}. We conclude that the orbit length of toric promotion on all graphs constructed by a single bridge sum of any combination of trees and complete graphs is given by $N(N-1)$, where $N$ is the number of vertices in the bridge sum graph. In particular, we show that the orbit length of toric promotion on graphs of this form is independent of the initial labeling.

\begin{theorem}
\label{THM:completebridgecomplete}
Let $K_1$ and $K_2$ be two complete graphs on $n_1$ and $n_2$ vertices, respectively. Let $G$ be a bridge sum of $K_1$ and $K_2$. Then all orbits of toric promotion on $G$ have length 
$(n_1+n_2)(n_2+n_1-1)$.
\end{theorem}

\begin{theorem}
 Let $G$ be a bridge sum (at any two vertices) of a tree on $m$ vertices and the complete graph on $n$ vertices. Then all orbits of toric promotion on $G$ have length $(m+n)(m+n-1)$.
\label{THM:TreeBridgeComplete}
\end{theorem}

Finally, we consider the {\color{blue} corona product} of a complete graph with any tree. 

\begin{definition} 
Let $G_1=(V_1,E_1)$ and $G_2=(V_2,E_2)$ be two graphs, where $V_1=\{v_1,\ldots, v_\nu\}$ and $v'\in V_2$. Then the corona product $G_1\odot G_2(v')$ is the graph obtained by taking $\nu$ copies of $G_2$, say ${^1G_2},\ldots,{^{\nu}G_2}$, where the copy of vertex $v'$ in ${^iG_2}$ is denoted $v'_i$, then adding the edge $\{v_i,v'_i\}$ for all $i\in \{1,\ldots,n\}$. In other words, if $^iV_2$ denotes the vertex set of $^iG_2$ and $^iE_2$ denotes the edge set of  $^iG_2$, then $G_1\odot G_2(v')$ is the graph with vertex set $V=V_1\cup {^1V_2} \cup \cdots \cup {^{\nu}V_2}$ and edge set $E=E_1\cup {^1E_2} \cup \cdots \cup {^{\nu}E_2} \cup \{v_1,v'_1\} \cup \cdots \cup \{v_{\nu},v'_{\nu}\}$.
\end{definition}

The corona product can be interpreted as a multiple-bridge sum. Meaning, the corona product $G_1\odot G_2(v')$ is the graph obtained by taking the bridge sum $G_1(v_i)\Bridge G_2(v')$ for all $v_i\in V_1$. We show an example of a corona product in \Cref{Fig:CoronaProd}.

\begin{figure}[H]
\begin{center}
    \includegraphics[width=.9\linewidth]{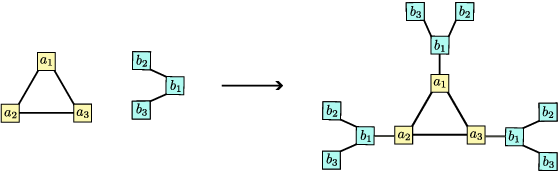}
\end{center}
\vspace{.3cm}
\caption{Let $G_1$ be the cycle graph on $3$ vertices and let $G_2$ be a tree on $3$ vertices. On the left, we show $G_1$ (yellow) and $G_2$ (blue). On the right, we show $G_1\odot G_2(b_1)$.}
\label{Fig:CoronaProd}
\end{figure}

We describe the orbit length of toric promotion on graphs constructed via the corona product, where $G_1$ is any complete graph and $G_2$ is any tree. We show that the orbit length of toric promotion on graphs of this form does not depend on the initial labeling.

\begin{theorem}
    Let $T=(V_T,E_T)$ be any tree on $m$ vertices, and let $K$ be the complete graph on $n$ vertices. All orbits of toric promotion on $ K\odot T(v')$, where $v'$ is any vertex in $V_T$, have size $(nm+n)(nm+n-1)$.
\label{THM:Corona}

\end{theorem}
Notice that this is again $N(N-1)$, where here $N$ denotes the number of vertices in $K\odot T(v')$.
\subsection{Outline}
In \Cref{Bridge sums with Simple Graphs}, we first prove \Cref{THM:CompleteGraphs}. Then, we define two objects, introduced by Adams, Defant, and Striker in \cite{ADS2024}, called {\color{blue}stone diagrams} and {\color{blue}coin diagrams}, which give an alternate but equivalent definition of toric promotion and provide a structure to formally study the orbits of toric promotion. We then prove Theorems~\ref{THM:SimpleBridgeTree} and \ref{THM:SimpleBridgeComplete}. In \Cref{orbitlengths}, we prove \Cref{THM:completebridgecomplete,THM:TreeBridgeComplete,THM:Corona}. Finally, in \Cref{Sec:FutureWorks}, we provide two directions for future work, with conjectures in each direction.

\section{Bridge sums of trees and complete graphs with simple graphs}
\label{Bridge sums with Simple Graphs}

Consider $K$, the complete graph on $n$ vertices. We will show that all orbits of toric promotion on $K$ have length $n$.

\begin{proof}[Proof of \Cref{THM:CompleteGraphs}]
    Let $(\sigma_0,i_0)$ be an initial state of toric promotion. To perform toric promotion, we look at the vertices labeled by $i$ and $i+1$. Since $K$ is a complete graph, for all iterations of toric promotion, the labels $i$ and $i+1$ will be on adjacent vertices. So, $\sigma$ will remain constant throughout the orbit. Therefore, the state of toric promotion after $I$ iterations is given by $(\sigma_0, i_0+I)$. After $I=n$ iterations, the state will be $(\sigma_0, i_0+n)$, or equivalently $(\sigma_0, i_0)$. 
\end{proof}

\subsection{Coin and Stone Diagrams}

Adams, Defant, and Striker \cite{ADS2024} introduce an alternative definition of toric promotion using stone and coin diagrams. In this subsection, we will define these objects. 

Consider a graph $G=(V,E)$, where $|V|=\nu$, and an initial state of toric promotion $(\sigma_0,i_0)$. To construct a coin diagram, place a coin on the vertex $\sigma_0^{-1}(i_0)$ of $G$.

To build the associated stone diagram, we first draw the corresponding cycle graph, denoted $\mathsf{Cycle}_\nu$, which is the cycle graph with vertex set $\mathbb{Z}/ \nu \mathbb{Z}$, whose vertices are arranged in clockwise cyclic order. The formal symbol $\mathbf{v}_i$ is called the {\color{blue}replica} of vertex $ v_i$, and is used to encode the labeling $\sigma$. That is, we place the replica $\mathbf{v}_i$ on the vertex $\sigma(v_i)$ of $\mathsf{Cycle}_\nu$. Additionally, we place a stone on vertex $i$ of $\mathsf{Cycle}_\nu$. We denote the stone diagram of a state $(\sigma, i) \in \Omega_G$ by $\SD(\sigma, i)$.

Furthermore, in \cite{ADS2024}, Adams Defant and Striker define the {\color{blue}cyclic shift} operator on the set of states $\text{cyc}:\Omega_G \rightarrow \Omega_G$ by $$\text{cyc}(\sigma, i)= (\text{cyc}(\sigma),i+1),$$ where $\text{cyc}(\sigma)$ is the map that sends $\sigma(v)$ to $\sigma(v)+1$ for all vertices $v\in V$. They define cyc on a stone diagram by $$\text{cyc}(\SD(\sigma, i))=\SD(\text{cyc}(\sigma),i+1).$$ A stone diagram $\SD'$ is called a {\color{blue}cyclic rotation} of a stone diagram $\SD$, if $\SD'=\text{cyc}^k(\SD)$ for some integer $k$. In other words, $\SD'$ is a cyclic rotation of $\SD$, if $\SD'$ is obtained by fixing the graph $\textsf{Cycle}_\nu$ in the plane and rotating all replicas (and the stone) by $k$ positions clockwise, for some integer $k$. See \Cref{EX:cyc} for a cyclic rotation. Additionally, we write $(\sigma(v_j),\sigma(v_k))\circ \sigma$ to denote the transposition that swaps the labels on the vertices $v_j$ and $v_k$ composed with $\sigma$. Notice that this is different from  $(j,k)\circ \sigma$, which denotes the transposition that swaps the labels $j$ and $k$ composed with $\sigma$.

\begin{figure}[htbp]
    \centering
\includegraphics[width=.6\linewidth]{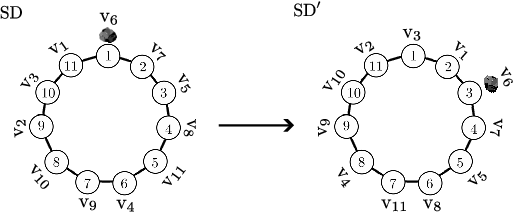}
\vspace{.3cm}
\caption{The stone diagram $\SD'$ (on the right) is a cyclic rotation of the stone diagram $\SD$ (on the left) by two positions clockwise, or $k=2$. In other words, $\SD'=\text{cyc}^2(\SD)$.}
\label{EX:cyc}
\end{figure}

We now interpret toric promotion as a set of rules on the stone and coin diagrams. That is, we perform toric promotion by looking at the replica that sits on the stone $\mathbf{v}_j$, and the replica that sits one position clockwise of the stone $\mathbf{v}_k$; or equivalently, the replica of the vertex of $G$ labeled by $i$ and the replica of the vertex of $G$ labeled by $i+1$. If the vertices $v_j$ and $v_k$
\begin{itemize}
    \item are adjacent in $G$, we slide the stone from under the replica $\mathbf{v}_j$ to under the replica $\mathbf{v}_k$, and we move the coin to the vertex $v_k$;
    \item otherwise, we swap the positions of the replicas $\mathbf{v}_j$ and $\mathbf{v}_k$, leaving the stone under the replica $\mathbf{v}_j$ (you can imagine the replica $\mathbf{v}_j$ surfing on the stone as it slides one position clockwise), and we leave the coin on vertex $v_j$.
\end{itemize}

If the coin is on a vertex $v_j$ at time $t$ and the coin is on a different vertex $v_k$ at time $ t+1$, then we say the coin moved from $v_j$ to $v_k$ at time $t$ and that the vertex $v_k$ received the coin at time $t+1$.

We can now track the orbits of toric promotion using stone and coin diagrams. In the remainder of this article, we will refer to this definition of toric promotion, as it will ease exposition in the remaining proofs. We give an example of an orbit of toric promotion using stone and coin diagrams in \Cref{EX:toric promotionStoneandCoin}.

\begin{figure}[h]
\centering
\includegraphics[width=\linewidth]{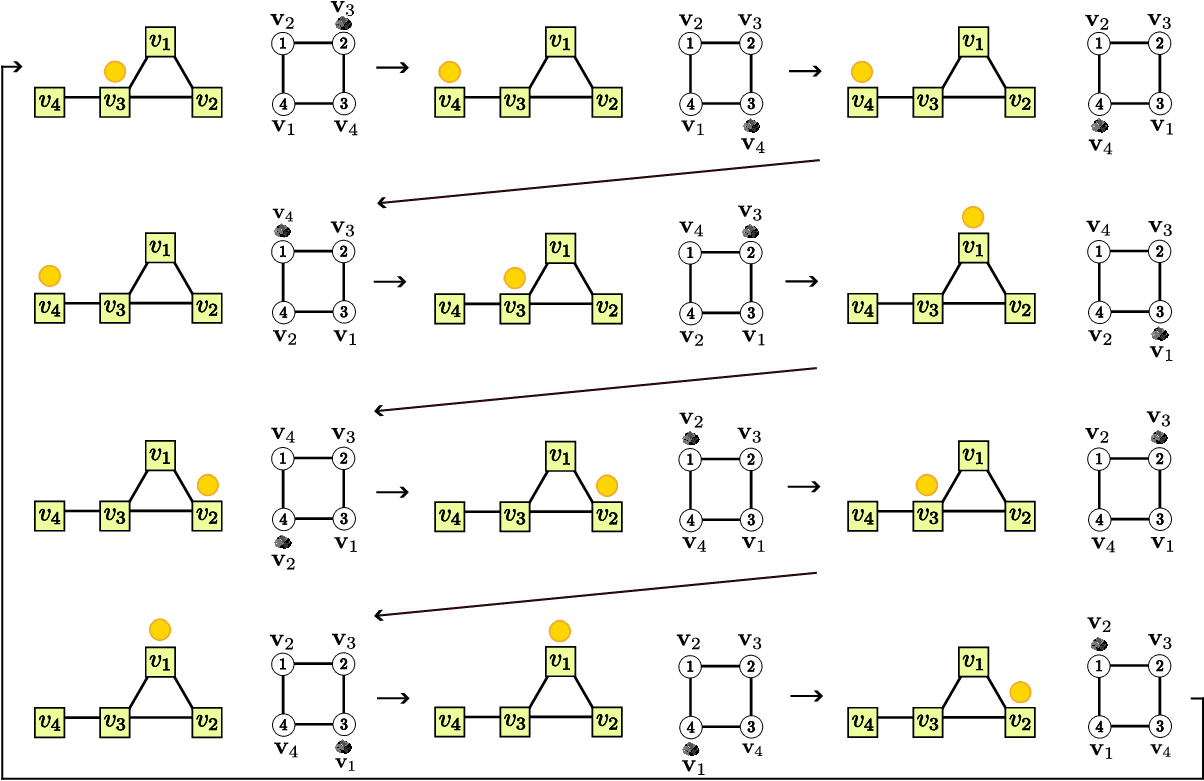}
\caption{Let $G$ be the graph on $4$ vertices shown in yellow. We show the orbit of toric promotion on $G$ containing the initial state $(4123,2)$ in terms of stone and coin diagrams. Note that this is the same orbit as in \Cref{EX: toric promotion orbit}.}
\label{EX:toric promotionStoneandCoin}
\end{figure}

\subsection{Trees Bridged with Simple Graphs}
Suppose we have a tree $T=(V_T,E_T)$ with vertex set $V_T=\{a_1,\ldots, a_m\}$ and a connected simple graph $S=(V_S,E_S)$ with vertex set $V_S=\{b_1,\ldots,b_\nu\}$. Let $G=T(a_i)\Bridge S(b_j)$ be a bridge sum of $T$ and $S$ at any two vertices $a_i$ and $b_j$. Let us denote the vertex set of $G$ with $V'$ and $|V'|=m+\nu$ with $N$.

The goal of this subsection is to prove \Cref{THM:SimpleBridgeTree}. That is, we want to show that the length of the orbit of toric promotion on $G$ containing the state $(\sigma,i)$ only depends on $m$, $S$, and $\phi$, the restriction of $\sigma$ to $V_S$. In particular, the orbit length of toric promotion on $G$ does not depend on the specific choice of $T$, nor the restriction of $\sigma$ to $V_T$.

Isomorphic graphs have orbits of equal size. So, without loss of generality, we can choose a naming of the vertices in $V'$. Let the vertex $a_i$ be named $p_m$ and let all other vertices in $V_T\subset V'$ be named bijectively with $\{p_1,\ldots,p_{m-1}\}$. Let $b_j$ be named $s_\nu$ and all other vertices in $V_S \subset V'$ be named bijectively with  $\{s_1,\dots,s_{\nu-1}\}$.

We will refer to the induced subgraph on $\{p_1,\ldots,p_{m}\}$ as $\mathcal{T}$, and the induced subgraph on $\{s_1,\ldots,s_\nu\}$ as $\mathcal{S}$. Then the edge $\{s_\nu,p_{m}\}$ is a bridge connecting $\mathcal{T}$ and $\mathcal{S}$, which we will call $\mathcal{B}$. 

This proof relies on showing that a lemma from \cite{ADS2024} (see \Cref{ADStreelemmma}) extends to the special case where the edge in question is $\mathcal{B}$. Before we state their lemma, let us define $\eta_{l,l'}$, a parameter Adams, Defant, and Striker use in the statement of \Cref{ADStreelemmma}. 

Given two adjacent vertices $v_l$ and $v_{l'}$ of a simple graph, let $S^{l,l'}$ be the set of vertices strictly closer to $v_{l'}$ than $v_l$ (including $v_{l'}$), where the distance from a vertex $v_x$ to a vertex $v_y$ is defined by the number of edges in the shortest path from $v_x$ to $v_y$. Furthermore, let $\eta_{l,l'}=|S^{l,l'}|$.

In \cite{ADS2024}, the authors give the following lemma for an $m$-vertex tree, $T=(V,E)$.

\begin{lemma} [\cite{ADS2024}]
Let $\{v_{l}, v_{l'}\}$ be an edge of $T$, and let $t$ be a time at which the coin moves from $v_l$ to $v_{l'}$. The first time after $t$ at which the coin moves from $v_{l'}$ to $v_l$ is $t+\eta_{l,l'}(m-1).$ Moreover, for all $v_j\in S^{(l,l')}$ and $v_k\in V$ with $j\not=k,$ there is a unique time in the interval $[t+1,t+\eta_{l,l'}(m-1)]$ at which $\vv_j$ sits on the stone and $\vv_k$ sits one position clockwise of $\vv_j$. In addition, $\SD_{t+\eta_{l,l'}(m-1)+1}$ is a cyclic rotation of  $\SD((\sigma(v_{l}),\sigma(v_{l'}))\circ\sigma_t,i_t+1)$.
\label{ADStreelemmma}
\end{lemma}

Since $G$ is not necessarily a tree, \Cref{ADStreelemmma} does not always apply to the edge $\mathcal{B}$. In the following lemma (see \Cref{LEM:TreeBridgeSimple}), we show that \Cref{ADStreelemmma} does in fact hold in the special case when the edge in question is $\mathcal{B}$. In general terms, we will show that each time the coin crosses $\mathcal{B}$ from $s_\nu$ to $p_m$, the coin will travel in a depth-first search, going along all branches of $\mathcal{T}$ before returning to $\mathcal{S}$. Specifically, the coin will spend exactly $m(N-1)$ time steps on $\mathcal{T}$ before returning to $\mathcal{S}$. We will show that during this breadth-first search, the coin spends exactly $N-1$ time steps on each vertex in $\mathcal{T}$. Furthermore, we will show that each time the coin enters $\mathcal{T}$, the stone diagram at the time the coin moves from $s_{\nu}$ to $p_m$ is a cyclic rotation of the stone diagram immediately after the next time the coin moves from $p_m$ to $s_{\nu}$, with the positions of the replicas $\mathbf{v}_{s_\nu}$ and $\mathbf{v}_{p_m}$ swapped. 

Notice that since $\mathcal{B}$ is a bridge, we have that $S^{s_\nu,p_{m}}=V_T$. It follows that, $\eta_{s_\nu,p_{m}}=m$. We now give the precise statement of the adapted lemma and its proof. 

\begin{lemma}
Consider $\mathcal{B}=\{s_\nu,p_m\}$, and let $t$ be a time at which the coin moves from $s_\nu$ to $p_m$. The first time after $t$ at which the coin moves from $p_m$ to $s_\nu$ is $t+m(N-1).$ Moreover, for all $v_j\in V_T$ and $v_k\in V'$ with $j\not=k,$ there is a unique time in the interval $[t+1,t+m(N-1)]$ at which $\vv_j$ sits on the stone and $\vv_k$ sits one position clockwise of $\vv_j$. In addition, $\SD_{t+m(N-1)+1}$ is a cyclic rotation of $\SD((\sigma_t(s_\nu),\sigma_t(p_{m}))\circ\sigma_t,i_t+1)$ by $\nu$ spaces clockwise.
\label{LEM:TreeBridgeSimple}
\end{lemma}

The proof is by induction on $\eta_{s_\nu,p_m}=m$, and mirrors the proof structure of \Cref{ADStreelemmma} in \cite{ADS2024}. We give an example of \Cref{LEM:TreeBridgeSimple} in \Cref{Fig:TreeBridgeSimpleLemmaEx}.

\begin{figure}[htbp]
  \centering
   \includegraphics[width=.87\linewidth]{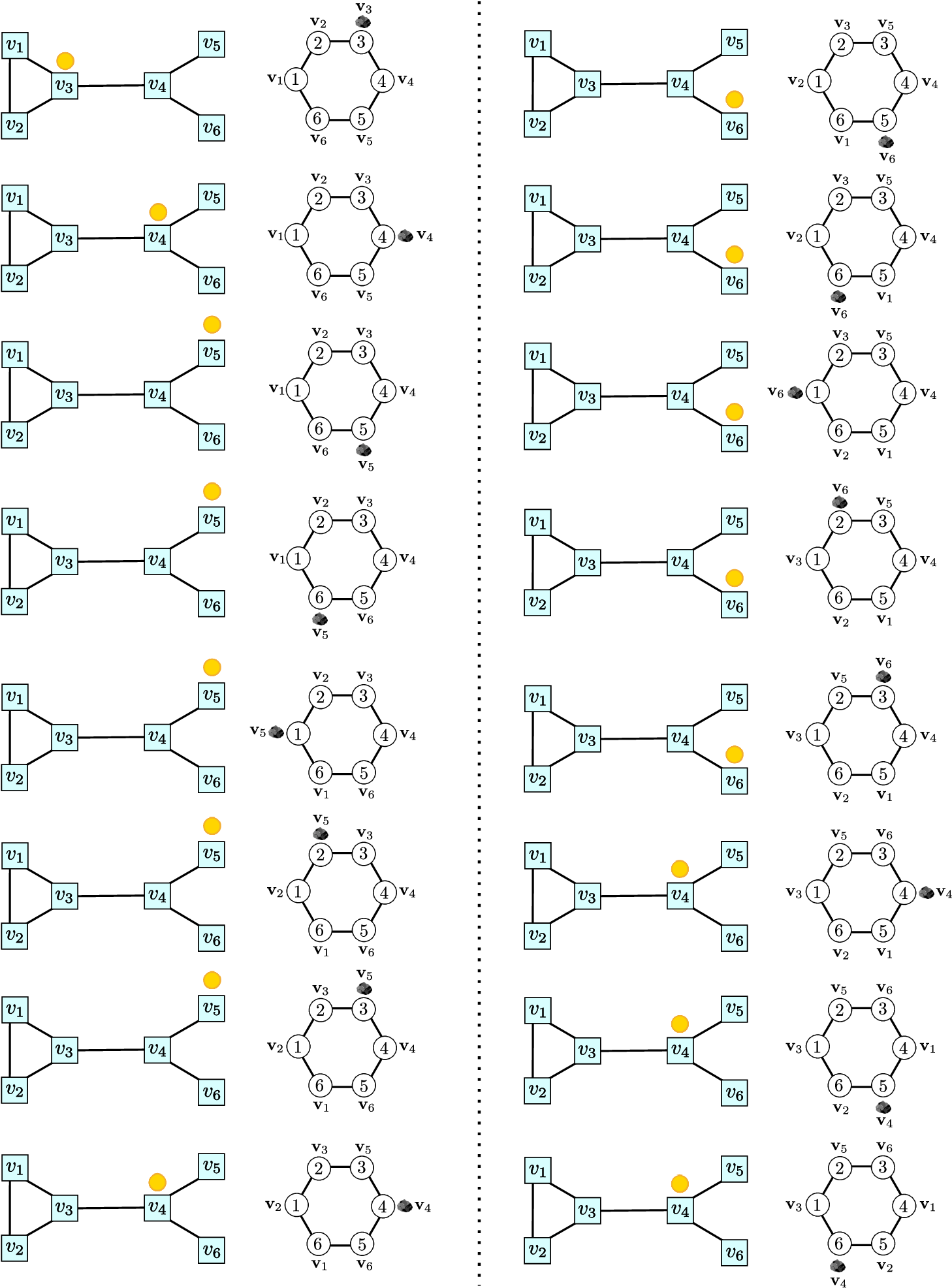}
 \caption{We show an example of \Cref{LEM:TreeBridgeSimple} on the bridge sum of the cycle graph on $3$ vertices and a tree on $3$ vertices, where $\mathcal{B}=\{v_3,v_4\}$. The top left image shows time $t$, when the coin crosses $v_3$ to $v_4$. The next image (below) shows time $t+1$. We have $F_1=\{ \emptyset \}$, $\textbf{F}_1= \{ \emptyset \}$, $f_1=0$, $F_2=\{ \emptyset \}$, $\textbf{F}_2= \{ \emptyset \}$, $f_2=0$, $F_3=\{ v_1,v_2 \}$, $\textbf{F}_3= \{ \textbf{v}_1,\textbf{v}_2 \}$, and $f_3=2$. If we were to perform toric promotion on the bottom right image, the coin would be on $v_3$ in the resulting coin diagram. So, the coin spends $15$ time steps in $\mathcal{T}$.} 
 \label{Fig:TreeBridgeSimpleLemmaEx}
\end{figure}

\begin{proof}
We give a proof by induction on $\eta_{s_\nu,p_m}=m$. First, assume $m=1$, then $V_T=\{p_{1}\}$, and $p_1$ has only one neighbor, $s_\nu$. At time $t+1$, the stone sits on $\mathbf{v}_{p_1}$ and $\mathbf{v}_{s_\nu}$ sits one position counterclockwise of $\mathbf{v}_{p_1}$. Since $\mathbf{v}_{p_1}$ sits on the stone, $\mathbf{v}_{p_1}$ can only move clockwise. Moreover, $s_\nu$ is the only neighbor of $p_1$. So, at time $t+1$, the replica $\mathbf{v}_{p_1}$ starts sliding through the replicas on the clockwise path from $\mathbf{v}_{p_1}$ to $\mathbf{v}_{s_{\nu}}$. There are $N-2$ replicas on this path. So, $\mathbf{v}_{p_1}$ will slide clockwise exactly $N-2$ positions on $\mathsf{Cycle}_N$. At time $t+1+(N-2)=t+(N-1)$, the replica $\mathbf{v}_{p_1}$ sits on the stone and the replica $\mathbf{v}_{s_\nu}$ sits one position clockwise of $\mathbf{v}_{p_1}$. It follows that at time $t+(N-1)$, the stone slides from under $\mathbf{v}_{p_1}$ to under $\mathbf{v}_{s_\nu}$. Therefore, the coin sits on vertex $p_1$ during the interval $[t+1,t+(N-1)]$, and at time $t+(N-1)$, the coin crosses $\mathcal{B}$ from $p_1$ to $s_\nu$.

Since the replica $\mathbf{v}_{p_1}$ sat on the stone and slid $N-2$ positions around $\mathsf{Cycle}_N$, it is clear that there is a unique time in the interval $[t+1,t+(N-1)]$ when $\mathbf{v}_{p_1}$ sat on the stone and each $\mathbf{v}_j\not= \textbf{v}_{p_{1}}$ sat one position clockwise of $\mathbf{v}_{p_1}$. Furthermore, in the interval $[t+1,t+(N-1)]$, each $\mathbf{v}_j$ for $j\not= \{p_{1},s_{\nu}\}$ was moved one position counterclockwise by the replica $\mathbf{v}_{p_{1}}$. Therefore, $\SD_{t+(N-1)}$ is a cyclic rotation of $\SD((\sigma_t(s_\nu),\sigma_t(p_{1}))\circ \sigma_t,i_t+1)$ by one position counterclockwise, or equivalently, $N-1=\nu$ position clockwise. This concludes the base case.

Now suppose $m\geq 2$. Let $d$ be the degree of $p_m$, and let $\mathfrak{N}[p_{m}]=\{v_{\alpha(1)},\ldots, v_{\alpha(d)}, p_m\}$ be the closed neighborhood of $p_{m}$, where $v_{\alpha(1)},\ldots,v_{\alpha(d)}$ are the neighbors of $p_m$ indexed so that their replica appear in clockwise cyclic order in $\SD_{t+1}$ and so that $\alpha(0)=\alpha(d)=s_\nu$. For $1\leq j \leq d$, let $F_j$ be the set of vertices in $V'\setminus \mathfrak{N}[p_{m}]$ whose replicas appear on the clockwise path from $\mathbf{v}_{\alpha(j-1)}$ to $\mathbf{v}_{\alpha(j)}$ in $\SD_{t+1}$. Let $\textbf{F}_j$ be the set of replicas of vertices in $F_j$ and let $f_j=\lvert F_j \rvert.$

Let $t_0=t+1$. At time $t_0$, the stone (carrying the replica $\mathbf{v}_{p_m}$) starts to slide clockwise through the replicas in $\textbf{F}_1.$ For each $v_k\in F_1$, there is a unique time in the interval $[t_0,t_0+f_1]$ at which $\mathbf{v}_{p_{m}}$ sits on the stone and $\mathbf{v}_k$ sits one position clockwise of $\mathbf{v}_{p_{m}}$. At time $t_0+f_1$, the coin moves from $p_{m}$ to $v_{\alpha(1)}$. 

We want to apply our inductive hypothesis to the edge $\{p_m,v_{\alpha(1)}\}$. Let us show that it can be applied. Since $\mathcal{T}$ is a tree, every edge in $\mathcal{T}$ is a bridge. So, the edge $\{p_m,v_{\alpha(1)}\}$ is a bridge between a simple graph and a tree. Specifically, the induced subgraph on $S^{p_m,v_{\alpha(1)}}$ is a tree, the induced subgraph on $V' \setminus S^{p_m,v_{\alpha(1)}}$ is a simple graph, and $\{p_m,v_{\alpha(1)}\}$ is a bridge connecting these two subgraphs. Hence, we can apply our inductive hypothesis to $\{p_m,v_{\alpha(1)}\}$. 

The first time after $t_0+f_1$ at which the coin moves from $v_{\alpha(1)}$ to $p_{m}$ is $(t_0+f_1)+\eta_{p_m,v_{\alpha(1)}}(N-1)$. Additionally, for all $v_j\in S^{p_m,v_{\alpha(1)}}$ and $v_k \in V'$ with $j\not= k$, there is a unique time in the interval \[ [t_0+f_1+1,(t_0+f_1)+\eta_{p_m,v_{\alpha(1)}}(N-1)] \] at which $\mathbf{v}_j$ sits on the stone and $\mathbf{v}_k$ sits one position clockwise of $\mathbf{v}_j$. Finally, $$\SD_{t_0+f_1+\eta_{p_m,v_{\alpha(1)}}(N-1)+1}$$ is a cyclic rotation of $$\SD((\sigma_{t_0+f_1}(v_{p_m}),\sigma_{t_0+f_1}(v_{\alpha(1)}))\circ \sigma_{t_0+f_1},i_{t_0+f_1}+1).$$

At time $(t_0+f_1)+\eta_{p_m,v_{\alpha(1)}}(N-1)+1$, the coin is sitting on $p_m$. Let $t_1=(t_0+f_1)+\eta_{p_m,v_{\alpha(1)}}(N-1)+1$. At time $t_1$, the stone (carrying the replica $\mathbf{v}_{p_{m}}$) starts to slide through the replica in $\textbf{F}_2$. At time $t_1+f_2$ the stone slides from under $\mathbf{v}_{p_{m}}$ to under $\mathbf{v}_{\alpha(2)}$. So, the coin moves from $p_{m}$ to $v_{\alpha(2)}$ at time $t_1+f_2$. For each $v_k\in F_2$, there is a unique time in the interval $[t_1,t_1+f_2]$ at which $\mathbf{v}_{p_{m}}$ sits on the stone and $\mathbf{v}_{k}$ sits one position clockwise of $\mathbf{v}_{p_{m}}$.

If the degree of $p_m=2$, then $v_{\alpha(2)}=s_\nu$. So, the first time after $t$ at which the coin moves from $p_m$ to $s_\nu$ is 
\begin{align*}
    t_1+f_2&=t_0+f_1+\eta_{p_m,v_{\alpha(1)}}(N-1)+1+f_2\\
    &=t_0+f_1+(m-1)(N-1)+1+f_2\\
    &=t+1+f_1+(m-1)(N-1)+1+f_2\\
    &=t+\sum_{i=1}^nf_i+d+(m-1)(N-1)\\
    &=t+(\nu-1)+d+(m-1)(N-1)\\
    &=t+(N-1)+(m-1)(N-1)\\
    &=t+m(N-1).\\
\end{align*}

Now, consider $d\geq 3$, then $\eta_{p_m,v_{\alpha(2)}} < \eta_{s_\nu,p_m}$. Once again, we can apply our inductive hypothesis to find that the first time after $t_1+f_2$ at which the coin moves from $v_{\alpha(2)}$ to $p_m$ is $t_1+f_2+\eta_{p_m,v_{\alpha(2)}}(N-1)$. Additionally, for all $v_j\in S^{p_m,v_{\alpha(2)}}$ and $v_k \in V'$ with $j\not= k$, there is a unique time in the interval \[[t_1+f_2+1,t_1+f_2+\eta_{p_m,v_{\alpha(2)}}(N-1)]\] at which $\mathbf{v}_j$ and $\mathbf{v}_k$ sits one position clockwise of $\mathbf{v}_j$. Moreover, $$\SD_{t_1+f_2+\eta_{p_m,v_{\alpha(2)}}(N-1)+1}$$ is a cyclic rotation of 

\begin{align*}
 \label{SD}
    \SD(\sigma_{t_1+f_2}(p_m),\sigma_{t_1+f_2}(v_{\alpha(2)})) \circ \sigma_{t_1+f_2}, i_{t_1+f_2}+1).
\end{align*}

Let $t_2=t_1+f_2+\eta_{p_m,v_{\alpha(2)}}(N-1)+1$. We can continue this process for $r \in \{0,1,\ldots,d-1\}$, defining $t_r=t_{r-1}+f_r+\eta_{p_m,\alpha(r)}(N-1)+1$.
Then, time $t_{d-1}+f_d$ is the first time after $t$ at which the coin moves from $v_{p_m}$ to $s_\nu$. Note that
\begin{align*}
     t_{d-1}+f_d&=t+\sum_{i=1}^d f_i+(\eta_{p_m,v_{\alpha(1)}}+\cdots+\eta_{p_m,v_{\alpha(d-1)}})(N-1)+d\\
     &=t+\vert V' \setminus \mathfrak{N}[p_m] \rvert + (\eta_{p_n,v_{\alpha(1)}}+\cdots+\eta_{p_m,v_{\alpha(d-1)}}) (N-1)+d\\
     &=t+(N-1)+(\eta_{p_m,v_{\alpha(1)}}+\cdots+\eta_{p_m,v_{\alpha(d-1)}}) (N-1)\\
     &=t+(N-1)+(m-1)(N-1)\\
     &=t+m(N-1).
\end{align*}

This shows the first statement of the lemma; that is, if $t$ is a time when the coin crosses $\mathcal{B}$ from $s_{\nu}$ to $p_m$, the first time after $t$ at which the coin moves from $p_m$ to $s_{\nu}$ is $t+m(N-1)$. Moreover, the coin spends exactly $N-1$ time steps on each vertex in $V_T$.  

Furthermore, up to cyclic rotation, the stone diagram $\SD_{t_r}$ for $1\leq r \leq d-1$ is obtained from $\SD_{t_{r-1}}$ by sliding the stone (carrying $\mathbf{v}_{p_m}$) clockwise through the replica in $\textbf{F}_r\cup\{\textbf{v}_{\alpha(r)}\}$. It follows directly that, up to cyclic rotation, $\SD_{t+m(N-1)}$ is obtained from $\SD_{t_0}$ by sliding the stone (carrying $\mathbf{v}_{p_{m}}$) clockwise through the replica in 
$$\textbf{F}_1\cup \cdots \cup \textbf{F}_d \cup \{ \textbf{v}_{\alpha(1)}, \ldots , \textbf{v}_{d-1}\} = V' \setminus \{ \textbf{v}_{p_{m}}, \textbf{v}_{s_\nu}\}.$$ Hence, for all $v_k \in V_T$ and $j\not=k$, there is a unique time in $[t+1, t+m(N-1)]$ at which the replica $\mathbf{v}_k$ sits on the stone and $\mathbf{v}_j$ sits one position clockwise of $\mathbf{v}_k$.

Finally, $\SD_{t+m(N-1)+1}$ is obtained from $\SD_{t+m(N-1)}$ by sliding the stone one step clockwise so that it slides from underneath $\mathbf{v}_{p_{m}}$ to underneath $\mathbf{v}_{s_\nu}$. Thus $\SD_{t+m(N-1)+1}$ is a cyclic rotation of \[ \SD((\sigma_t(v_{p_{m}}),\sigma_t(v_{s_\nu}))\circ \sigma_t, i_t+1)\] by $\nu$ positions clockwise. This concludes the proof of \Cref{LEM:TreeBridgeSimple}.
\end{proof}

We conclude this subsection by proving \Cref{THM:SimpleBridgeTree}.

\begin{proof}[Proof of \Cref{THM:SimpleBridgeTree}]
We have shown that $\SD_{t+m(N-1)+1}$ is a cyclic rotation of $$\SD((\sigma_t(v_{p_{m}}),\sigma_t(v_{s_\nu}))\circ \sigma_t, i_t+1).$$ It follows that the effect of the coin traveling to $\mathcal{T}$ is the same, up to cyclic rotation, as if the vertices $s_\nu$ and $p_{m}$ were not adjacent. If the coin is on vertex $i$ of a graph $H$, made up of multiple connected components, it is clear that the labeling on components not containing the vertex $i$ has no effect on the orbit length of orbits containing the state $(\sigma, i)$. So, the labels on vertices in $\mathcal{T}$ have no effect on the orbit length of toric promotion on $G$. This concludes the proof of \Cref{THM:SimpleBridgeTree}.
\end{proof}

\subsection{Complete Graphs Bridged with Simple Graphs}
\label{Sec:CompleteBirdgeSimple}

Consider the complete graph $K=(V_K,E_K)$ with vertex set $V_K=\{a_1,\ldots, a_n\}$ and a connected simple graph $S=(V_S,E_S)$ with vertex set $V_S=\{b_1,\ldots,b_\nu\}$. Let $G=K(a_i)\Bridge S(b_j)$ be a bridge sum of $K$ and $S$, at any two vertices $a_i$ and $b_j$. Let us denote the vertex set of $G$ with $V'$, and let $|V'|=n+\nu$ be denoted $N$.

The goal of this subsection is to prove \Cref{THM:SimpleBridgeComplete}, that is, we want to show that the length of the orbit of toric promotion on $G$ containing the state $(\sigma,i)$ depends only on $n$, $S$, and $\phi$, the restriction of $\sigma$ to $V_S$. In particular, it does not depend on the restriction of $\sigma$ to $V_K$.

Once again, we can rename the vertices in $V'$. Let $a_i$ be named $k_n$, and let all vertices in $V_K\setminus a_i$ be named bijectively with the set $\{k_1, \ldots, k_{n-1}\}$. Furthermore, let $b_j$ be named $s_\nu$ and let all vertices in $V_S\setminus b_j$ be named bijectively with the set $\{s_1,\ldots,s_{\nu-1}\}$. We will refer to the induced subgraph on $\{k_1,\ldots, k_{n}\}$ as $\mathcal{K}$ and the induced subgraph on $\{s_1,\ldots,s_\nu\}$ as $\mathcal{S}$. Then, the edge $\{k_n,s_\nu\}$ is a bridge connecting $\mathcal{K}$ and $\mathcal{S}$, which we will refer to as $\mathcal{B}$.

First, we prove two lemmas that will be useful in the proof of \Cref{THM:SimpleBridgeComplete}. 

\begin{lemma}
\label{LEM:CompleteCWvsCCW}
    If $[t,t']$ is an interval where the coin is on $\mathcal{K}$, then during the interval $[t,t']$, replicas of vertices in $V_K$ can only move clockwise or remain in place, and replicas of vertices in $V_T$ can only move counterclockwise or remain in place. 
\end{lemma}
  
\begin{proof}
Let $v_\ell$ be a vertex in $V_K$. If the coin is on $v_\ell$, then the replica $\mathbf{v}_{\ell}$ is sitting on the stone. While $\mathbf{v}_{\ell}$ is sitting on the stone, $\mathbf{v}_{\ell}$ can only move clockwise. Furthermore, while $\mathbf{v}_{\ell}$ is sitting on the stone, the stone can either face a replica of a vertex not adjacent to $v_\ell$, or the stone can face a replica of a vertex adjacent to $v_\ell$. We will begin with the former.

Let $v_j$ be a vertex not adjacent to $v_\ell$; since $\mathcal{K}$ is complete, the vertex $v_j$ must be in $V_T$. If $\mathbf{v}_{\ell}$ sits on the stone and the stone faces $\mathbf{v}_j$, the stone (with replica $\mathbf{v}_{\ell}$) will swap places with $\mathbf{v}_j$. Hence, $\mathbf{v}_j$ will move counterclockwise, $\mathbf{v}_{\ell}$ will move clockwise, and all other replicas will remain in place as desired. 

Now consider the case where $\mathbf{v}_{\ell}$ is sitting on the stone and the stone is facing a replica of a vertex adjacent to $v_\ell$. If $v_j$ is adjacent to $v_\ell$, then $v_j$ is in either $V_T$ or $V_K$. 

Consider $v_j \in V_T$. If  $\mathbf{v}_{\ell}$ is sitting on the stone, and the stone is facing $\mathbf{v}_j$, then the stone will slide from underneath $\mathbf{v}_{\ell}$ to underneath $\mathbf{v}_j$. Since the replica $\mathbf{v}_j$ is now sitting on the stone and the coin is now sitting on $v_j$, the coin is no longer in $\mathcal{K}$. In the interval while the replica $\mathbf{v}_{\ell}$ sat on the stone, all replicas of vertices in $V_K$ either remained in place or moved clockwise. Furthermore, all replicas of vertices in $V_T$ either remained in place or moved counterclockwise as desired. 

Now consider $v_j \in V_K$. If $\mathbf{v}_{\ell}$ is sitting on the stone, and the stone is facing $\mathbf{v}_j$, then the stone will slide from underneath $\mathbf{v}_{\ell}$ to underneath $\mathbf{v}_j$. Since $v_j \in V_K$, the above arguments hold. This concludes the proof of \Cref{LEM:CompleteCWvsCCW}.
\end{proof}

The following lemma mirrors \Cref{LEM:TreeBridgeSimple}, adapted for the case of a complete graph. The reader may wish to refer to \Cref{EX:laps} throughout this proof. 

\begin{lemma}
Consider $\mathcal{B}=\{k_n,s_\nu \}$, and let $t$ be a time at which the coin moves from $s_\nu$ to $k_n$. The first time after $t$ at which the coin moves from $k_n$ to $s_\nu$ is $t+n(N-1)$. In addition, $\SD_{t+n(N-1)+1}$ is a cyclic rotation of $\SD((\sigma_t(k_n),\sigma_t(s_\nu))\circ \sigma_t,i_t+1)$ by $\nu$ spaces clockwise. 
\label{LEM:CompleteBridgeSimple}
\end{lemma}

\begin{proof}
Firstly, if $n=1$ or $n=2$, then $K$ is a tree. Therefore, by \Cref{LEM:TreeBridgeSimple}, the claim holds. For the rest of this proof, we assume $n\geq 3$.

At time $t$, the coin moves from $s_\nu$ to $k_n$, or equivalently, at time $t$, the stone slides from under $\mathbf{v}_{s_\nu}$ to under $\mathbf{v}_{k_n}$. Therefore, at time $t+1$ the replica $\mathbf{v}_{k_n}$ sits on the stone. Let $t'$ be the first time after $t$ when the coin moves from $k_n$ to $s_\nu$, or equivalently, the first time after $t$ when the stone slides from under $\mathbf{v}_{k_n}$ to under $\mathbf{v}_{s_\nu}$. We want to show that $t'=t+n(N-1)$.

Let $\mathbf{v}_{\alpha(0)}, \ldots, \textbf{v}_{\alpha(n-1)}$ be the replicas of vertices in $V_K$, indexed so that they appear in clockwise cyclic order in $\SD_{t}$, and so that $\alpha(0)=\alpha(n)=k_n$. Since the coin moves at time $t$, the labeling is constant from time $t$ to $t+1$. Thus the indices $\alpha(i)$ are constant for $[t,t+1]$.

Let $\mathbf{v}_{\beta(0)}, \ldots, \textbf{v}_{\beta(n-1)}$ be the replicas of vertices in $V_S$, indexed so that they appear in clockwise cyclic order in $\SD_{t}$, and so that $\beta(0)=\beta(\nu)=s_\nu$. Again, the coin moves at time $t$. So, the labeling is constant from time $t$ to $t+1$. Thus the indices $\beta(j)$ are constant for $[t,t+1]$.

Additionally, let us define the sets $A_0,\ldots,A_{n-1}$, so that the set $A_0$ is the set of replicas on the clockwise cyclic path from $\mathbf{v}_{\alpha(0)}$ to $\mathbf{v}_{\alpha(1)}$ in $\SD_{t}$ (not including $\mathbf{v}_{\alpha(0)}$ and $\mathbf{v}_{\alpha(1)}$), and for $0\leq r \leq n-2$, the set $A_r$ is the set of replicas on the clockwise cyclic path from $\mathbf{v}_{\alpha(r)}$ to $\mathbf{v}_{\alpha(r+1)}$ in $\SD_{t}$ (not including $\mathbf{v}_{\alpha(r)}$ and $\mathbf{v}_{\alpha(r+1)}$). Finally, let $A_{n-1}$ be the set of replicas on the clockwise cyclic path from $\mathbf{v}_{\alpha(n-1)}$ to $\mathbf{v}_{\alpha(0)}$ (not including $\mathbf{v}_{\alpha(n-1)}$, $\mathbf{v}_{\alpha(0)}$, or $\mathbf{v}_{\beta(0)}$). Then, reading clockwise starting from $\mathbf{v}_{\alpha(0)}$, the cyclic order of replicas in $\SD_t$ is $$\mathbf{v}_{\alpha(0)}, A_0, \textbf{v}_{\alpha(1)}, A_1,\ldots, {\textbf{v}_{\alpha(n-2)}}, A_{n-2}, {\textbf{v}_{\alpha(n-1)}}, A_{n-1}, \textbf{v}_{\beta(0)}.$$

We will evaluate how the stone diagram evolves during the interval $[t,t']$. By assumption, at time $t+1$, the stone sits on $\mathbf{v}_{\alpha(0)}$. Furthermore, at time $t+1$ the replica $\mathbf{v}_{\alpha(0)}$ sits on the stone, and $\mathbf{v}_{\beta(0)}$ sits one position counterclockwise of $\mathbf{v}_{\alpha(0)}$. In comparison, at time $t'$, the replica $\mathbf{v}_{\alpha(0)}$, sits on the stone and $\mathbf{v}_{\beta(0)}$ sits one position clockwise of $\mathbf{v}_{\alpha(0)}$. So, during the interval $[t+1,t']$, the relative movement of $\mathbf{v}_{\alpha(0)}$ and $\mathbf{v}_{\beta(0)}$ towards each other sums to $N-2$ spaces.

 During the interval $[t+1,t']$, the coin is on $\mathcal{K}$. By \Cref{LEM:CompleteCWvsCCW}, replicas of vertices in $V_K$ can only move clockwise or remain in place, and replicas of vertices in $V_S$ can only move counterclockwise or remain in place. So, during $[t+1,t']$, the replica $\mathbf{v}_{\alpha(0)}$ can only move clockwise, and only moves if it moves through the replica of a vertex that is not adjacent to $k_n$. There are $n-1$ replicas of this form. So, during $[t+1,t']$ the replica $\mathbf{v}_{\alpha(0)}$ can only move $n-1$ spaces clockwise. Similarly, during $[t+1,t']$, the replica $\mathbf{v}_{\beta(0)}$ can only move counterclockwise, and only moves if the replica of a vertex that is not adjacent to $s_\nu$ sits on the stone and the stone moves through $\mathbf{v}_{\beta(0)}$. Only replicas of the vertices in $V_K$ sit on the stone during $[t+1,t']$ and there are $\nu-1$ vertices in $V_K$ that are not adjacent to $s_\nu$. So, the replica $\mathbf{v}_{\beta(0)}$ can only move counterclockwise $\nu-1$ spaces during $[t+1,t']$. Since the total movement of $\mathbf{v}_{\alpha(0)}$ and $\mathbf{v}_{\beta(0)}$ during $[t+1,t']$ must sum to $N-2$ spaces, and $\mathbf{v}_{\alpha(0)}$ can only move $n-1$ clockwise spaces and $\mathbf{v}_{\beta(0)}$ can only move $\nu-1$ counterclockwise spaces, it is clear that during the interval $[t+1,t']$, the replica $\mathbf{v}_{\alpha(0)}$ will move exactly $n-1$ spaces clockwise and the replica $\mathbf{v}_{\beta(0)}$ will move exactly $\nu-1$ spaces counterclockwise. It follows that $t'$ is exactly the time when $\mathbf{v}_{\alpha(0)}$ has moved through all $A_r$. We will show that $t'=t+n(N-1)$.
 
Let $t+1=t_0$. At time $t_0$, the replica $\mathbf{v}_{\alpha(0)}$ receives the stone. Let $t_1$ be the first time after $t_0$ when $\mathbf{v}_{\alpha(0)}$ receives the stone. In general, let $t_{\ell}$ be the $\ell^{th}$ time after $t_0$ when $\mathbf{v}_{\alpha(0)}$ receives the stone. We call an interval of the form $[t_{\ell},t_{\ell+1}-1]$ a lap; specifically, we call $[t_{\ell},t_{\ell+1}-1]$ lap $\ell$. Furthermore, let $t_\ell^{(i)}$ be the time during lap $\ell$, when the replica $\mathbf{v}_{\alpha(i)}$ receives the stone. Then $t_\ell=t_\ell^0$ for all $\ell$.

We will now evaluate how the stone diagram evolves during lap $0$, or equivalently, the interval $[t_0^{(0)},t_1^{(0-1)}]$. At time $t_0=t_0^{(0)}$, the replica $\mathbf{v}_{\alpha(0)}$ sits on the stone and starts sliding through the replica in $A_0$. At time $t^{(0)}_0+\lvert A_0\rvert$, the stone slides from underneath $\mathbf{v}_{\alpha(0)}$ to under $\mathbf{v}_{\alpha(1)}$. So, at time $t^{(0)}_0+\lvert A_0 \lvert + 1$,  the replica $\mathbf{v}_{\alpha(1)}$ receives the stone. Hence, $t^{(0)}_0+\lvert A_0 \lvert + 1=t_0^{(1)}$. At time $t_0^{(1)}$, the replica $\mathbf{v}_{\alpha(1)}$ sits on the stone and the stone starts sliding through the replicas in $A_1$. It follows that for $1 \leq r \leq n-1$, we have that $t_0^{(r)}=t^{(r-1)}_0+ \lvert A_{r-1} \rvert +1$. At time $t^{(n-1)}_0$, the replica $\mathbf{v}_{\alpha(n-1)}$ sits on the stone. Since $\mathbf{v}_{\alpha(0)}$ moved through the set $A_0$, the set of replicas on the clockwise path from $\mathbf{v}_{\alpha(n-1)}$ to $\mathbf{v}_{\alpha(0)}$ in $\SD_{t^{(n-1)}_0}$ is $A_{n-1}\cup \mathbf{v}_{\beta(0)} \cup A_0$. Hence, at time $t^{(n-1)}_0 + \lvert A_{n-1}\rvert + 1 + \lvert A_0 \rvert$, the stone slides from under $\mathbf{v}_{\alpha(n-1)}$ to under $\mathbf{v}_{\alpha(0)}$. Therefore, $t^{(n-1)}_0 + \lvert A_{n-1}\rvert + 1 + \lvert A_0 \rvert+1$ is the first time after $t^{(0)}_0$ when the replica $\mathbf{v}_{\alpha(0)}=\mathbf{v}_{k_n}$ receives the stone; so, $t_1^{(0)}=t^{(n-1)}_0 + \lvert A_{n-1}\rvert + 1 + \lvert A_0 \rvert+1$. Solving for $t_1^{(0)}$, we have 
\begin{align*}
t_1^{(0)}&=t^{(n-1)}_0 + \lvert A_{n-1}\rvert + 1 + \lvert A_0 \rvert+1\\
&=(t^{(n-2)}_0+\lvert A_{n-2} \rvert +1)+ \lvert A_{n-1}\rvert + 1 + \lvert A_0 \rvert+1\\
    &=t_0^{(0)}+\sum_{i=0}^{n-1}\lvert A_{i}\rvert+|A_0|+n+1\\
    &=t_0^{(0)}+\nu+|A_0|+n\\
    &=t_0^{(0)}+N+|A_0|.
\end{align*}

We have shown that $t_1^{(0)}=t_0^{(0)}+N+ \lvert A_0 \rvert$. Therefore, lap $0$ will take $N+ \lvert A_0\rvert$ time steps.  

We will now define new notation to evaluate the stone diagrams for laps $1$ through $n-2$. Let $A'_0$ be the set of replicas on the clockwise cyclic path from $\mathbf{v}_{\alpha(0)}$ to $\mathbf{v}_{\alpha(1)}$ in $\SD_{t_1^{(0)}}$; then,  $A'_0=A_1$. In general, let $A'_r$, for $0\leq r \leq n-1$, be the set of replicas on the clockwise cyclic path from $\mathbf{v}_{\alpha(r)}$ to $\mathbf{v}_{\alpha(r+1)}$ in $\SD_{t_1^{(0)}}$. Then, for $0\leq r \leq n-3$, the set $A'_r$ is equal to $A_{r+1}$, the set $A'_{n-2}$ is equal to $A_{n-1}\cup \mathbf{v}_{\beta(0)} \cup A_0$, and the set $A'_{n-1}$ is empty. In other words, the clockwise cyclic order of replicas in $\SD_{t_1^{(0)}}$, starting from $\mathbf{v}_{\alpha(0)}$, is given by

$$\mathbf{v}_{\alpha(0)}, A_0', \mathbf{v}_{\alpha(1)}, A_1', \ldots, \mathbf{v}_{\alpha(n-2)}, A_{n-2}', \mathbf{v}_{\alpha(n-1)}.$$

We can now evaluate how the stone diagrams evolve during lap $1$, or equivalently, during the interval $[t_1^{(0)}, t_2^{(0-1)}]$. At time $t_1^{(0)}$, the replica $\mathbf{v}_{\alpha(0)}$ sits on the stone, and the stone starts to move through the replicas in $A'_0$. At time $t_1^{(0)}+|A'_0|$, the stone slides from under $\mathbf{v}_{\alpha(0)}$ to under $\mathbf{v}_{\alpha(1)}$; hence, $t_1^{(1)}=t_1^{(0)}+|A'_0|+1$. For $1\leq r \leq n-1$, it follows that $t_1^{(r)}=t_1^{(r-1)}+|A'_{r-1}|+1$. At time $t_1^{(n-1)}$, the replica $\mathbf{v}_{\alpha(n-1)}$ sits on the stone. Since $\mathbf{v}_{\alpha(0)}$ moved through $A'_0$ during lap $1$, the set of replicas on the clockwise cyclic path from $\mathbf{v}_{\alpha(n-1)}$ to $\mathbf{v}_{\alpha(0)}$ in $\SD_{t_1^{(n-1)}}$ is $A'_{n-1} \cup {A'_0}$. Thus, $t_1^{(n-1)}+|A'_{n-1} \cup {A'_0}|+1$ is the second time after $t^{(0)}_0$ when the replica $\mathbf{v}_{\alpha(0)}=\mathbf{v}_{k_n}$ receives the stone. So, $t_2^{(0)}=t_1^{(n-1)}+|A'_{n-1} \cup {A'_0}|+1$. Recall, that $A'_{n-1}$ is empty, so $t_2^{(0)}=t_1^{(n-1)}+|{A'_0}|+1$. We can solve for $t_2^{(0)}$. Thus,
\begin{align*}
    t_2^{(0)}&=t_1^{(n-1)}+|{A'_0}|+1\\
   & =(t_1^{(n-2)}+|A'_{(n-2)}|+1)+|{A'_0}|+1\\
    &=t_1^{(0)}+\sum_{i=1}^{n-2}|A'_i|+|A'_{0}|+n\\
    &=t_1^{(0)}+\sum_{i=1}^{n-1}|A_i|+|A_{1}|+n+1\\
    &=t_1^{(0)}+\nu+n+|A_{1}|\\
   & =t_1^{(0)}+N+|A_1|.
\end{align*}
We can conclude that lap $1$ will take $N+|A_1|$ time steps. 

For $1\leq \ell \leq n-2$, the set of replicas on the clockwise cyclic path from $\mathbf{v}_{n-1}$ to $\mathbf{v}_{\alpha(0)}$ in $\SD_{t_\ell^{(0)}}$ will be empty. During lap $\ell$, for $1\leq \ell \leq n-2$, the replica $\mathbf{v}_{n-1}$ will move through $\mathbf{v}_{\beta(j)}$ if any only if the replica $\mathbf{v}_{\alpha(0)}$ moves through $\mathbf{v}_{\beta(j)}$. Furthermore, during lap $\ell$, for $1 \leq \ell \leq n-2$, the replica $\mathbf{v}_{\alpha(0)}$ moves through the set $A'_{\ell-1}$. It follows that, during lap $\ell$, for $1 \leq \ell \leq n-2$, the replica $\mathbf{v}_{n-1}$ also moves through the set $A'_{\ell-1}$. Additionally, during lap $\ell$, for $1 \leq \ell \leq n-2$, the replica $\mathbf{v}_{\alpha(i)}$, for $1 \leq i \leq n-2$, moves through $A'_{(\ell-1)+i}$. Therefore, lap $\ell$, for $1 \leq \ell \leq n-2$, will take $N+|A'_{\ell-1}|$ time steps, or equivalently, $N+|A_\ell|$ time steps. We showed that lap $0$ will take $N+|A_0|$ time steps. Therefore, $t_{n-1}^{(0)}$ is given by 
\begin{align*}
    t_{n-1}^{(0)}=t_0^{(n-2)}+N+|A_{n-2}|.
\end{align*}
 At time $t_{n-1}^{(0)}$, the replica $\mathbf{v}_{\alpha(0)}$ sits on the stone and starts to move through the replica in $A_{n-1}$. At time $t_{n-1}^{(0)}+|A_{n-1}|$, the stone slides from under $\mathbf{v}_{\alpha(0)}$ to under $\mathbf{v}_{\beta(0)}$. Hence, time $t_{n-1}^{(0)}+|A_{n-1}|$ is the first time after $t_0^{(0)}$ when the stone slides from under $\mathbf{v}_{\alpha(0)}$ to under $\mathbf{v}_{\beta(0)}$. So, $t'=t_{n-1}^{(0)}+|A_{n-1}|$. We can solve for $t'$. We have,
\begin{align*}
    t'&=t_{n-1}^{(0)}+|A_{n-1}|\\
    &=t_0^{(0)}+(n-1)(N)+\sum_{i=1}^{n-1}|A_i|\\
   & =t_0^{(0)}+(n-1)(N)+(\nu-1)\\
    &=t_0^{(0)}+(n-1)+(n-1)(N-1)+(\nu-1)\\
   & =t+1+(n-1)+(n-1)(N-1)+(\nu-1)\\
    &=t+(N-1)+(n-1)(N-1)\\
   & =t+n(N-1).
\end{align*}
We have now shown the first statement of the lemma; that is, the first time after $t$ at which the coin moves from $k_n$ to $s_\nu$ is $t+n(N-1)$. 

Now we will show the second statement, that is, we will show that $\SD_{t+n(N-1)+1}$ is a cyclic rotation of $\SD((\sigma_t(k_n),\sigma_t(s_\nu))\circ\sigma_t,i_t+1)$ by $\nu$ positions clockwise. This follows from the evaluation of the stone diagram from time $t$ to $t+n(N-1)$. 

We showed that in the interval $[t_0^{(0)},t+n(N-1)]$ each replica $\mathbf{v}_{\alpha(i)}$ for $1 \leq i \leq n-1$ moved through exactly $\nu$ replicas, namely the replicas in the set $A_0\cup \cdots \cup A_{n-1}\cup \textbf{v}_{\beta(0)}$. Thus, each $\mathbf{v}_{\alpha(i)}$ for $1 \leq i \leq n-1$ moved exactly $\nu$ spaces in the clockwise direction during the interval $[t_0^{(0)},t+n(N-1)]$. Additionally, we showed that $\mathbf{v}_{\alpha(0)}$ moved through exactly $\nu-1$ replicas during $[t_0^{(0)},t+n(N-1)]$, namely the replicas in the set $A_0\cup \cdots \cup A_{n-1}$. Hence, $\mathbf{v}_{\alpha(0)}$ moved $\nu-1$ spaces in the clockwise direction during the interval $[t_0^{(0)},t+n(N-1)]$. Moreover, each set $A_0\cup \cdots \cup A_{n-1}$ was moved through by $n$ replicas, namely all $\mathbf{v}_{\alpha(i)}$. Therefore, each replica in $A_0\cup \cdots \cup A_{n-1}$ moved $n$ spaces in the counterclockwise direction during the interval $[t_0^{(0)},t+n(N-1)]$. Finally, $\mathbf{v}_{\beta(0)}$ was moved through by $n-1$ replicas, namely all $\mathbf{v}_{\alpha(i)}$ for $1\leq i\leq n-1$. Therefore, the replica $\mathbf{v}_{\beta(0)}$ moved $n-1$ spaces in the counterclockwise direction during the interval $[t_0^{(0)},t+n(N-1)]$. Notice that moving $x$ spaces clockwise is the same as moving $N-x$ spaces counterclockwise. During the interval $[t_0^{(0)},t+n(N-1)]$, all replicas (other than $\mathbf{v}_{\alpha(0)}$ and $\mathbf{v}_{\beta(0)}$) moved either $n$ spaces clockwise or $\nu$ spaces counterclockwise. Since $N-\nu=n$, we can conclude that in $\SD_{t+n(N-1)+1}$, all replicas (other than $\mathbf{v}_{\alpha(0)}$ and $\mathbf{v}_{\beta(0)}$) are $\nu$ spaces 
clockwise of their positions in $\SD_t$. Furthermore, the replica $\mathbf{v}_{\alpha(0)}=\mathbf{v}_{k_n}$ is $\nu-1$ spaces clockwise of its position in $\SD_t$, and the replica $\mathbf{v}_{\beta(0)}=\mathbf{v}_{s_\nu}$ is $n-1$ spaces counterclockwise ($\nu+1$ spaces clockwise) of its position in $\SD_t$. The stone diagram $\SD((\sigma_t(k_n),\sigma_t(s_\nu))\circ\sigma_t,i_t+1)$ is obtained from $\SD_t$ by swapping the positions of $\mathbf{v}_{k_n}$ (and the stone) and $\mathbf{v}_{s_\nu}$. This moves $\mathbf{v}_{k_n}$ (and the stone) one space clockwise and $\mathbf{v}_{s_\nu}$ one space counterclockwise. Therefore, we have shown that $\SD_{t+n(N-1)+1}$ is a cyclic rotation of $\SD((\sigma_t(k_n),\sigma_t(s_\nu))\circ\sigma_t,i_t+1)$ by $\nu$ positions clockwise.
\end{proof}

\begin{example}
In \Cref{Fig:LapsExample}  we show an orbit of toric promotion on the bridge sum of a simple graph on $6$ vertices and the complete graph on $5$ vertices, where ${k_n}=v_7$,  ${s_\nu}=v_6$, and $\mathcal{B}=\{v_6,v_7\}$. We show the stone and coin diagrams at time $t$, when the coin crosses $\mathcal{B}$ from $v_6$ to $v_7$. We show the rest of this orbit divided into laps. Each row corresponds to a lap. In the first row, we show the first lap (lap $0$). The sets $A_i$ and $\mathbf{v}_6$ are highlighted in different colored ovals: $A_0=\{\textbf{v}_5\}$ is in red, $A_1=\{\textbf{v}_3\}$ is in orange,  $A_2=\{\textbf{v}_4\}$ is in yellow, $A_3=\{\textbf{v}_2\}$ is in green, $A_4=\{\textbf{v}_1\}$ is in blue, and $\mathbf{v}_6$ is in pink. In laps $1,2$ and $3$, the sets $A'_i$ are highlighted in colored rectangles:  $A_0'=\{\textbf{v}_3\}$ is in green, $A_1'=\{\textbf{v}_4\}$ is in pink,  $A_2'=\{\textbf{v}_2\}$ is in blue, and $A_3'=\{\textbf{v}_1,\textbf{v}_6, \textbf{v}_5\}$ is in yellow. In the last row, $\mathbf{v}_7$ moves through $A_4$ to complete the orbit.
\label{EX:laps}
\end{example}
\newpage

\begin{figure}[H]
    \centering
    \includegraphics[width=0.985\linewidth]{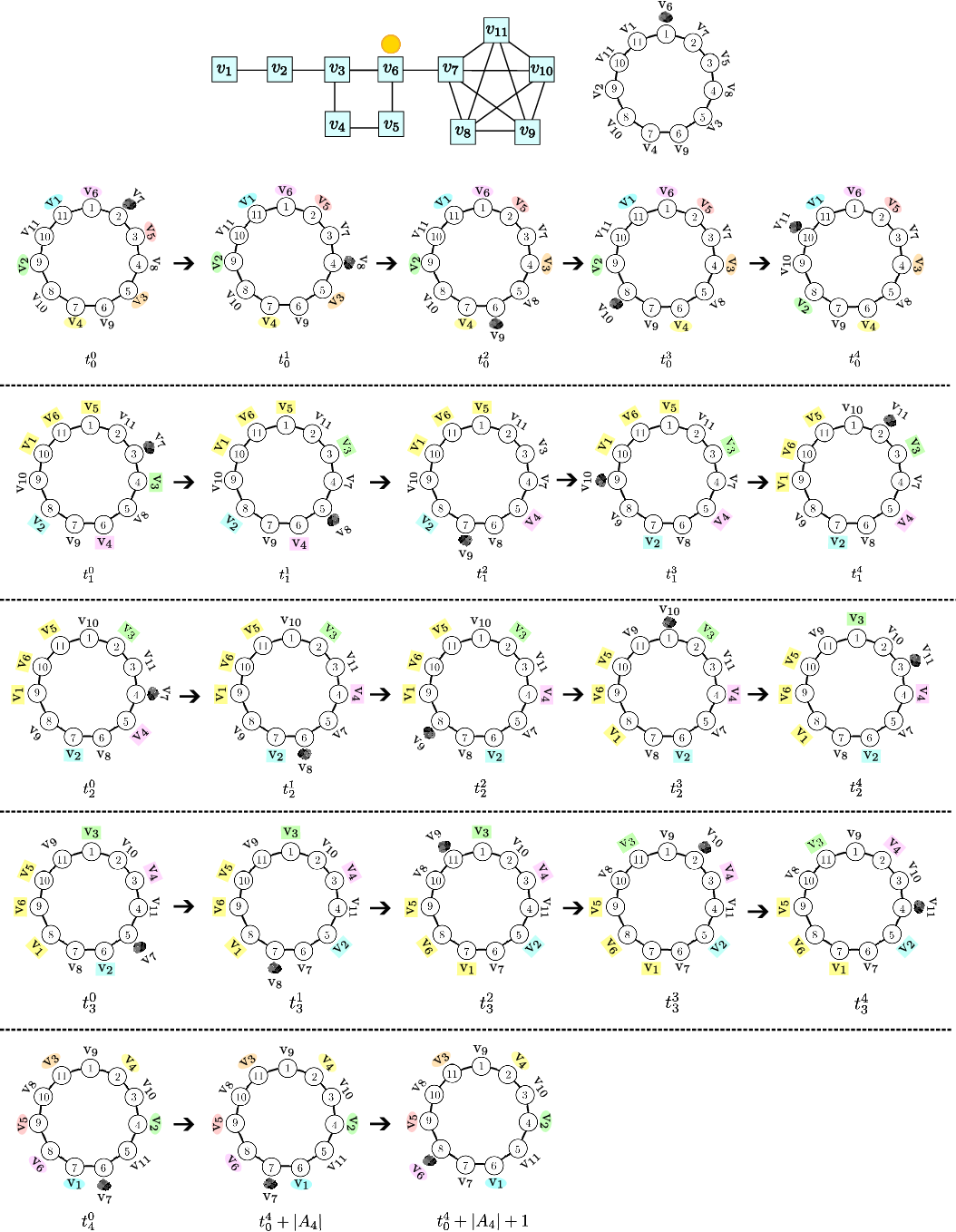}
    \caption{We give an example of \Cref{LEM:CompleteBridgeSimple}.}
    \label{Fig:LapsExample}
\end{figure}

Now we will prove \Cref{THM:SimpleBridgeComplete}.
\begin{proof}[Proof of \Cref{THM:SimpleBridgeComplete}]
We have shown that $\SD_{t+n(N-1)+1}$ is a cyclic rotation of $\SD((\sigma_t(k_n),\sigma_t(s_\nu))\circ\sigma_t,i_t+1)$. Thus, up to cyclic rotation, the effect of the coin traveling to $\mathcal{K}$ is the same as if the vertices $k_n$ and $s_\nu$ were not adjacent. If the coin is on vertex $i$ of a graph $H$, made up of multiple connected components, it is clear that the labeling on components not containing the vertex $i$ has no effect on the orbit length of orbits containing the state $(\sigma,i)$. So, the labels on vertices in $\mathcal{K}$ have no effect on the orbit length of toric promotion on $G$.
\end{proof}

\section{Orbit Lengths of Toric Promotion}
\label{orbitlengths}

The orbit length of toric promotion on trees and complete graphs individually does not depend on the initial labeling. In this section, we prove \Cref{THM:completebridgecomplete,THM:TreeBridgeComplete,THM:Corona}.  We conclude that the orbit length of toric promotion on a single bridge sum of any combination of trees and complete graphs and the orbit length of toric promotion on the corona product of a complete graph with a tree is given by $N(N-1)$, where $N$ is the number of vertices in the promoted graph. 

\subsection{Orbit Lengths of Toric Promotion on Certain Bridge Sums}

Consider two complete graphs $K_1$ and $K_2$ on $n_1$ and $n_2$ vertices, respectively. Suppose $K_1(v_1)\Bridge K_2(v_2)$ is a bridge sum of $K_1$ and $K_2$ at any two vertices $v_1$ and $v_2$. We will refer to the edge $\{v_1,v_2\}$ as $\mathcal{B}$. The graph $K_1(v_1)\Bridge K_2(v_2)$ has $n_1+n_2$ vertices, let $N=n_1+n_2$

We will now prove \Cref{THM:completebridgecomplete}; that is, we will show that all orbits of toric promotion on $K_1(v_1)\Bridge K_2(v_2)$ have length $N(N-1)$. The proof follows from \Cref{LEM:CompleteBridgeSimple}. 

\begin{proof}[Proof of \Cref{THM:completebridgecomplete}]
Let $t$ be a time when the coin crosses $\mathcal{B}$ from $v_1$ to $v_2$. By \Cref{LEM:CompleteBridgeSimple}, the first time after $t$ when the coin crosses $\mathcal{B}$ from $v_2$ to $v_1$ is $t+n_2(N-1)$. Additionally, $\SD_{t+n_2(N-1)+1}$ is a cyclic rotation of $\SD((\sigma_t(v_1),\sigma_t(v_2)\circ \sigma_t,i_t+1)$ by $n_1$ positions clockwise.

Since the coin crosses $\mathcal{B}$ from $v_2$ to $v_1$ at time $t+n_2(N-1)$, by \Cref{LEM:CompleteBridgeSimple} the first time after $t+n_2(N-1)$ when the coin moves across $\mathcal{B}$ from $v_1$ to $v_2$ is 
\begin{align*}
    t+n_1(N-1)+n_2(N-1)=t+N(N-1).
\end{align*}
Additionally, $\SD_{t+N(N-1)+1}$ is a cyclic rotation of $$\SD((\sigma_{t+n_2(N-1)}(v_1),\sigma_{t+n_2(N-1)}(v_2))\circ \sigma_{t+n_2(N-1)}, i_{t+n_2(N-1)}+1)$$ by $n_1$ positions clockwise.

The coin moves at time $t+n_2(N-1)$. So, the cyclic order of replicas in $\SD_{t+n_2(N-1)}$ is the same as the cyclic order of replicas in $\SD_{t+n_2(N-1)+1}$. Additionally, the coin sits on $v_1$ at time $t+n_2(N-1)+1$. Furthermore, the coin moves at time $t+N(N-1)$. So, the cyclic order of replicas in $\SD_{t+(N)(N-1)}$ is the same as the cyclic order of replicas in $\SD_{t+(N)(N-1)+1}$. Additionally, the coin sits on $v_1$ at time $t+N(N-1)$. Therefore, $\SD_{t+N(N-1)}$ is a cyclic rotation of $$\SD((\sigma_{t+n_2(N-1)+1}(v_1),\sigma_{t+n_2(N-1)+1}(v_2))\circ \sigma_{t+n_2(N-1)+1}, i_{t+n_2(N-1)+1}-1)$$ by $n_2$ positions clockwise. It follows that $\SD_{t+N(N-1)}$ is a cyclic rotation of $\SD_t$ by $N$ positions clockwise. So, $\SD_{t+N(N-1)}=\SD_t$. 
\end{proof}

Now consider any tree $T$ on $m$ vertices, and the complete graph $K$ on $n$ vertices. Suppose $T(v_1)\Bridge K(v_2)$ is a bridge sum of $T$ and $K$ at any two vertices $v_1$, $v_2$.  We will refer to the edge $\{v_1,v_2\}$ as $\mathcal{B}$. The graph $T(v_1)\Bridge K(v_2)$ has $m+n$ vertices, let $N=m+n$.

We will now prove \Cref{THM:TreeBridgeComplete}; that is, we will show that all orbits of toric promotion on $G$ have length $N(N-1)$. This proof follows the structure of the proof of \Cref{THM:TreeBridgeComplete} exactly, except here we use \Cref{LEM:CompleteBridgeSimple} and \Cref{LEM:TreeBridgeSimple}. \Cref{Fig:TreeBridgeCompleteExample} illustrates an example of the following argument. 

\begin{figure}[htbp]
     \centering
    \includegraphics[width=\linewidth]{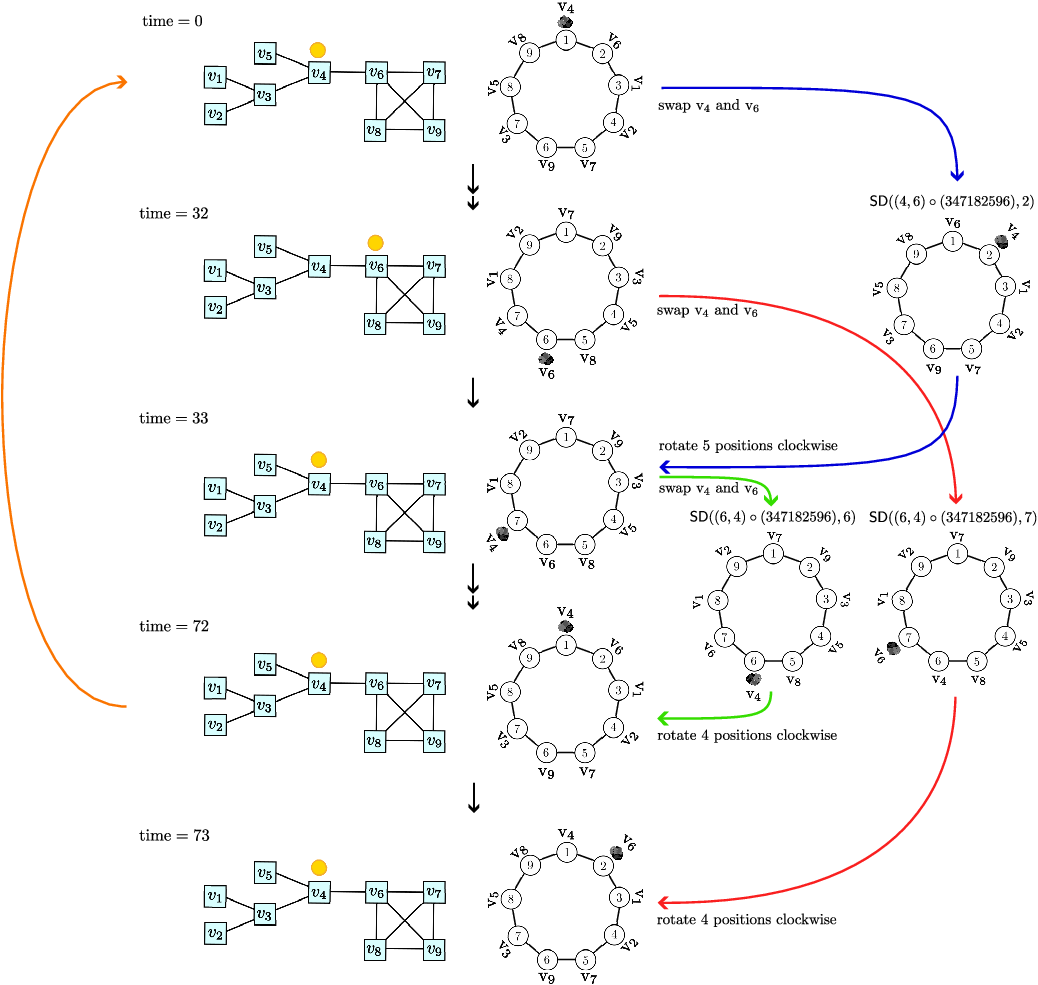}
\caption{We show the orbit of toric promotion on the bridge sum of a tree on $5$ vertices and the complete graph on $4$ vertices. This orbit has length $72=9\cdot8$ as expected. The blue arrow shows how $\SD_{33}$ is obtained from $\SD_{0}$, employing \Cref{LEM:CompleteBridgeSimple}. The red arrow shows how $\SD_{73}$ is obtained from $\SD_{32}$, employing \Cref{LEM:TreeBridgeSimple}. The green arrow shows how $\SD_{72}$ is obtained from $\SD_{33}$. The orange arrow shows that $\SD_{0}=\SD_{72}$.}
\label{Fig:TreeBridgeCompleteExample}
\end{figure}

\begin{proof}
Let $t$ be a time when the coin crosses $\mathcal{B}$ from $v_1$ to $v_2$. By \Cref{LEM:CompleteBridgeSimple}, the first time after $t$ when the coin crosses $\mathcal{B}$ from $v_2$ to $v_1$ is $t+m(N-1)$. Additionally, $\SD_{t+m(N-1)+1}$ is a cyclic rotation of $\SD((\sigma_t(v_1),\sigma_t(v_2))\circ \sigma_t,i_t+1)$ by $n$ positions clockwise.

Since the coin crosses $\mathcal{B}$ from $v_2$ to $v_1$ at time $t+m(N-1)$, by \Cref{LEM:TreeBridgeSimple}, the first time after $t+m(N-1)$ when the coin moves across $\mathcal{B}$ from $v_1$ to $v_2$ is 
\begin{align*}
t+n(N-1)+m(N-1)=t+N(N-1).
\end{align*}
Additionally, $\SD_{t+N(N-1)+1}$ is a cyclic rotation of $$\SD((\sigma_{t+N(N-1)}(v_1),\sigma_{t+N(N-1)}(v_2))\circ \sigma_{t+N(N-1)}, i_{t+N(N-1)}+1)$$ by $m$ positions clockwise. 

The coin moves at time $t+m(N-1)$. So, the cyclic order of replicas in $\SD_{t+m(N-1)}$ is the same as the cyclic order of replicas in $\SD_{t+m(N-1)+1}$. Additionally, the coin sits on $v_1$ at time $t+m(N-1)+1$. Moreover, the coin moves at time $t+N(N-1)$. So, the cyclic order of replicas in $\SD_{t+N(N-1)}$ is the same as the cyclic order of replicas in $\SD_{t+N(N-1)+1}$. Additionally, the coin sits on $v_1$ at time $t+N(N-1)$. Therefore, $\SD_{t+N(N-1)}$ is a cyclic rotation of $$\SD((\sigma_{t+m(N-1)+1}(v_1),\sigma_{t+m(N-1)+1}(v_2))\circ \sigma_{t+m(N-1)+1}, i_{t+m(N-1)+1}-1)$$ by $m$ positions clockwise. It follows that $\SD_{t+N(N-1)}$ is a cyclic rotation of $\SD_t$ by $N$ positions clockwise. So, $\SD_{t+N(N-1)}=\SD_t$. 
\end{proof}

\subsection{Orbit Lengths of Toric Promotion on Certain Corona Products}

Consider $H=K \odot T(v')$, where $K$ is the complete graph on $n$ vertices, $T$ is any tree on $m$ vertices, and $v'$ is any vertex of $T$. It is clear that $H$ has $nm+n$ vertices. We will denote $nm+n$ by $N$. In this subsection, we prove \Cref{THM:Corona}; that is, we show that all orbits of toric promotion on $H$ have length $N(N-1)$. 

Let the vertex set of $K$ be $V_K=\{k_1,\ldots,k_n\}$. The graph $H=K \odot T(v')$ is constructed by taking the bridge sum $K(k_i) \Bridge T(v')$ for all $k_i\in V_K$. Let the copy of $T$ that is bridged with vertex $k_i$ of $V_K$ be denoted $T^{(i)}$. Let the vertex set of $T^{(i)}$ be $V_{T^{(i)}}=\{p^{(i)}_1,\ldots,p^{(i)}_m\}$. Furthermore, let the bridge connecting $K$ and $T^{(i)}$ be the edge $\{p^{(i)}_1,k_i\}$. We denote the induced subgraph on $\{p^{(i)}_1,k_i\}$ with $\mathcal{B}^{(i)}$. Finally, we denote the induced subgraph $V_K$ with $\mathcal{K}$ and the induced subgraph on $V_{T^{(i)}}$ with $\mathcal{T}^{(i)}$.

Without loss of generality, let $t$ be a time when the coin crosses $\mathcal{B}^{(1)}$ from $k_1$ to $p_1^{(1)}$. In the following proof, we first use \Cref{LEM:TreeBridgeSimple} to show that $t+m(N-1)$ is the first time after $t$ at which the coin crosses $\mathcal{B}^{(1)}$ from $p_1^{(1)}$ to $k_1$. Let $t'$ be the first time after $t$ at which the coin crosses $\mathcal{B}^{(1)}$ from $k_1$ to $p_1^{(1)}$. We show that in the interval $[t+m(N-1),t']$ the replicas $\mathbf{v}_{p_1^{(1)}}$ and $\mathbf{v}_{k_1}$ need to move a relative distance of $N-2$ spaces towards each other. Specifically, we show that $t'$ is exactly the time when the clockwise distance between $\mathbf{v}_{p_1^{(1)}}$ and $\mathbf{v}_{k_1}$ has changed from $N-2$ spaces to $0$ spaces. We adapt the lap-based argument used in the proof of \Cref{LEM:CompleteBridgeSimple} to show that $t'=t+N(N-1)$. Finally, we show that $\SD_t=\SD_{t+N(N-1)}$. The reader may wish to refer to \Cref{fig:corona} throughout the following argument. 

\begin{figure}[htbp]
    \centering
    \includegraphics[width=.96\linewidth]{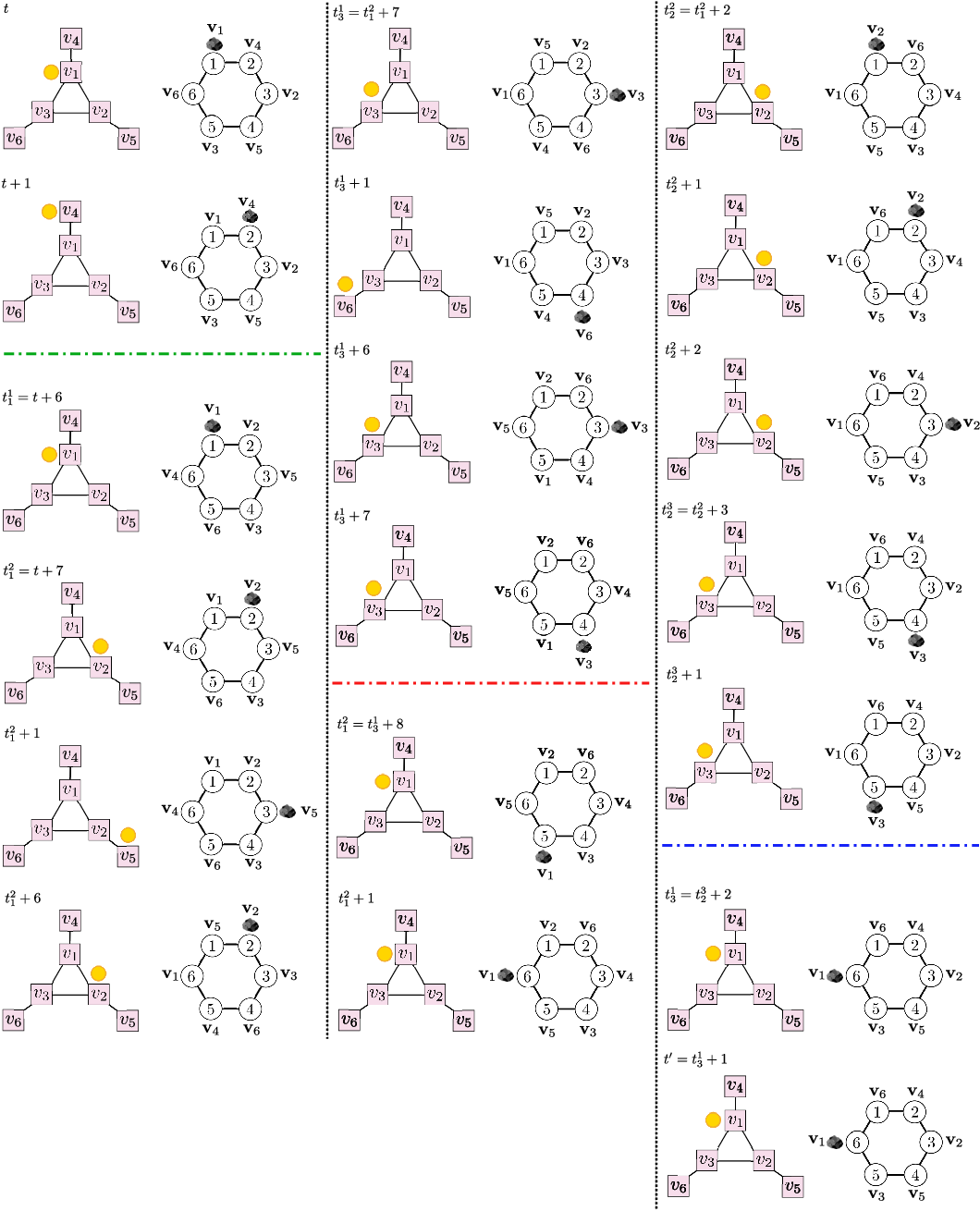}
    \caption{We show an example of an orbit of toric promotion on the corona product of the complete graph on $3$ vertices and the tree on $1$ vertex. The upper left image shows time $t$. The green horizontal line after $t+1$ marks the start of lap $1$. The red horizontal line after $t_3^{(1)}+7$ marks the start of lap $2$. The blue horizontal line after $t^{(3)}_2+1$ marks the end of lap $2$. In this example $k_{\alpha(1)}=v_1$, $k_{\alpha(2)}=v_2$, $k_{\alpha(3)}=v_3$, $A_1=\emptyset$, $A_2=\vv_5$, and $A_3=\vv_4,\vv_6$.}
  \label{fig:corona}
\end{figure}

\begin{proof}
Let $t$ be a time when the coin moves from $k_1$ to $p_1^{(1)}$, and let $t'$ be the first time after $t$ at which the coin moves from $k_1$ to $p_1^{(1)}$. We want to show that all orbits of toric promotion on $H$ have size $N(N-1)$. We will do this by first showing that $t'=t+N(N-1)$. Then,  we will show that $\SD_t=\SD_{t+N(N-1)}$. 

By \Cref{LEM:TreeBridgeSimple}, the first time after $t$ at which the coin moves from $p_1^{(1)}$ to $k_1$ is $t+m(N-1)$. Additionally, $\SD_{t+m(N-1)+1}$ is a cyclic rotation of $\SD((\sigma_t(k_1),\sigma_t(p_1^{(1)}))\circ \sigma_t, i_t+1)$ by $N-m$ spaces clockwise. 

At time $t+m(N-1)+1$, the replica $\mathbf{v}_{k_1}$ sits in the stone and $\mathbf{v}_{p_1^{(1)}}$ sits one position counterclockwise of $\mathbf{v}_{k_1}$. At time $t'$, the replica $\mathbf{v}_{k_1}$ sits on the stone and $\mathbf{v}_{p_1^{(1)}}$ sits one position clockwise of $\mathbf{v}_{k_1}$. Since $k_1$ and $p_1^{(1)}$ are adjacent, the replicas $\mathbf{v}_{k_1}$ and $\mathbf{v}_{p_1^{(1)}}$ cannot move through each other. So, in the interval $[t+m(N-1)+1 ,t']$ the replicas $\mathbf{v}_{k_1}$ and $\mathbf{v}_{p_1^{(1)}}$ must move a relative distance of $N-2$ spaces towards each other. In other words, during the interval $[t+m(N-1)+1 ,t']$, the sum of the movement of $\mathbf{v}_{p_1^{(1)}}$ in the clockwise direction and $\mathbf{v}_{k_1}$ in the counterclockwise direction must total to $N-2$ spaces.

Each time the coin exits $\mathcal{K}$, the coin must cross some bridge $\mathcal{B}^{(i)}$. By \Cref{LEM:TreeBridgeSimple}, if the coin crosses a bridge $\mathcal{B}^{(i)}$ from $k_i$ to $p_1^{(i)}$ at a time $\tau$, the first time after $\tau$ at which the coin crosses $\mathcal{B}^{(i)}$ from $p_1^{(i)}$ to $k_i$ is $\tau+m(N-1)$, and the cyclic order of replicas in $\SD_{\tau+m(N-1)+1}$ is the same as the cyclic order of replicas in $\SD_{\tau}$, except the replicas $\mathbf{v}_{p_1^{(i)}}$ and $\mathbf{v}_{k_i}$ have switched places. By assumption, in the interval $[t+m(N-1)+1, t']$ the coin does not cross $\mathcal{B}^{(1)}$. Hence, the coin entering and exiting $\mathcal{K}$ during $[t+m(N-1)+1, t']$ does not change the relative positions of $\mathbf{v}_{p_1^{(1)}}$ and $\mathbf{v}_{k_1}$. So we can conclude that the relative distance between $\mathbf{v}_{p_1^{(1)}}$ and $\mathbf{v}_{k_1}$ can only change if the coin is in $\mathcal{K}$. 

By \Cref{LEM:CompleteCWvsCCW}, if the coin is in $\mathcal{K}$, replicas of vertices in $V_K$ can only move clockwise or remain in place, and replicas of vertices not in $V_K$ can only move counterclockwise or remain in place. Therefore, in the interval $[t+m(N-1)+1, t']$, the replica $\mathbf{v}_{k_1}$ can only move clockwise or remain in place. The replica $\mathbf{v}_{k_1}$ can only move if the stone is on $\mathbf{v}_{k_1}$ and a replica of a vertex not adjacent to $k_1$ sits one position clockwise of $\mathbf{v}_{k_1}$. There are exactly $nm-1$ vertices of this form. Hence, in the interval $[t+m(N-1)+1, t']$, the replica $\mathbf{v}_{k_1}$ can move a maximum of $nm-1$ spaces clockwise. Moreover, in the interval $[t+m(N-1)+1, t']$, the replica $\mathbf{v}_{p_1^{(1)}}$ can only move counterclockwise or remain in place. The replica $\mathbf{v}_{p_1^{(1)}}$ can only move if a replica of a vertex in $V_K$, which is not adjacent to $p_1^{(1)}$, sits on the stone and is one position counterclockwise of $\mathbf{v}_{p_1^{(1)}}$. There are exactly $n-1$ vertices of this form. So, $\mathbf{v}_{p_1^{(1)}}$ can move a maximum of $n-1$ spaces counterclockwise. Notice that $(n-1)+(nm-1)=N-2$. It follows that $t'$ is the time when $\mathbf{v}_{k_1}$ has moved through $nm-1$ other replicas and $\mathbf{v}_{p_1^{(1)}}$ has been moved through by $n-1$ other replicas. Since the stone sits on $\mathbf{v}_{k_1}$ at time $t'$, we can conclude that $t'$ is exactly the time when $\mathbf{v}_{k_1}$ has moved through $nm-1$ other replicas. 

In the remainder of this proof, we adapt the lap-based argument used in \Cref{LEM:CompleteBridgeSimple}. Let $\vv_{k_\alpha(1)}, \ldots, \vv_{k_\alpha(n)}$ be the replicas of vertices in $V_k$, indexed so that they appear in clockwise cyclic order in $\SD_{t+m(N-1)+1}$, and so that $\mathbf{v}_{k_{\alpha(1)}}=\mathbf{v}_{k_1}$. Let $t_\ell^{(1)}$ be the $\ell^{th}$ time after $t$ when the replica $\mathbf{v}_{k_{\alpha(1)}}$ receives the stone. So, $t+m(N-1)+1=t_1^{(1)}$. An interval of the form $[t_\ell^{(1)},t_{\ell+1}^{(1)}-1]$ is a called a \textit{lap}; specifically, the interval $[t_\ell^{(1)},t_{\ell+1}^{(1)}-1]$ is called lap $\ell$. Furthermore, let $t_\ell^{(i)}$ for $2\leq i \leq n$ be the first time after $t_\ell^{(1)}$ when the replica $\mathbf{v}_{k_{\alpha(i)}}$ receives the stone. In other words, $t_\ell^{(i)}$ is the first time during lap $\ell$ when $\mathbf{v}_{k_{\alpha(i)}}$ receives the stone. Moreover, let the replicas on the clockwise cyclic path from $\mathbf{v}_{k_{\alpha(n)}}$ to $\mathbf{v}_{k_{\alpha(1)}}$ (not including $\mathbf{v}_{k_{\alpha(n)}}$ and $\mathbf{v}_{k_{\alpha(1)}}$) in $\SD_{t_1^{(1)}}$ be the set $A_n$. For $1\leq i \leq n-1$, let the replicas on the clockwise cyclic path from $\mathbf{v}_{k_{\alpha(i)}}$ to $\mathbf{v}_{k_{\alpha(i+1)}}$ in $\SD_{t_1^{(1)}}$ be the set $A_i$. Then the clockwise cyclic order of replicas in $\SD_{t_1^{(1)}}$, starting from $\mathbf{v}_{k_{\alpha(1)}}$ is given by
$$\mathbf{v}_{k_{\alpha(1)}}, A_1, \mathbf{v}_{k_{\alpha(2)}}, A_2, \ldots,  \mathbf{v}_{k_{\alpha(n)}}, A_n.$$
Finally, let $p_1^{(\alpha(i))}$ be the vertex in $\mathcal{T}^{(\alpha(i))}$ that is adjacent to $k_{\alpha(i)}$.

Let us start by evaluating how the stone diagram evolves during lap $1$. At time $t_1^{(1)}$, the replica $\mathbf{v}_{k_{\alpha(1)}}$ sits on the stone and starts sliding through the replicas in $A_1$. At time $t_1^{(1)}+|A_1|$, the stone slides from under $\mathbf{v}_{k_{\alpha(1)}}$ to under $\mathbf{v}_{k_{\alpha(2)}}$. So, $t_1^{(2)}=t_1^{(1)}+|A_1|+1$. At time $t_1^{(2)}$, the stone sits on $\mathbf{v}_{k_{\alpha(2)}}$ and starts sliding through the replicas in $A_2$. We have two cases:
\begin{enumerate}
\item the replica $\mathbf{v}_{p_1^{(\alpha(2))}}$ is in $A_2$;
\item or, the replica $\mathbf{v}_{p_1^{(\alpha(2))}}$ is not in $A_2$.
\end{enumerate}
Let's start with case 1; that is the replica $\mathbf{v}_{p_1^{(\alpha(2))}}$ is in $A_2$. At time $t_1^{(2)}$, the stone sits on $\mathbf{v}_{k_{\alpha(2)}}$ and starts sliding through the replicas in $A_2$, until some time when the replica $\mathbf{v}_{p_1^{(\alpha(2))}}$ sits one space clockwise of $\mathbf{v}_{k_{\alpha(2)}}$. We will call this time $t^*$. At time $t^*$, the stone slides from under $\mathbf{v}_{k_{\alpha(2)}}$ to under $\mathbf{v}_{p_1^{(\alpha(2))}}$. By \Cref{LEM:TreeBridgeSimple}, the next time after $t^*$ when the stone slides from under $\mathbf{v}_{p_1^{(\alpha(2))}}$ to under $\mathbf{v}_{k_{\alpha(2)}}$ is $t^*+m(N-1)$. Additionally, $\SD_{t^*+m(N-1)+1}$ is cyclic rotation of $$\SD((\sigma_{t^*}(k_{\alpha(2)}),\sigma_{t^*}(p_1^{\alpha(2)}))\circ \sigma_{t^*}, i_{t^*}+1)$$ by $N-m$ spaces clockwise. At time $t^*+m(N-1)+1$, the stone sits on $\mathbf{v}_{k_{\alpha(2)}}$ and starts sliding through the remaining replicas in $A_2$. At time $t_1^{(2)}+m(N-1)+|A_2|$, the stone slides from under $\mathbf{v}_{k_{\alpha(2)}}$ to under $\mathbf{v}_{k_{\alpha(3)}}$. It follows that $t_1^{(3)}=t_1^{(2)}+m(N-1)+|A_2|+1$. This concludes case $1$.

Now let's examine case two; that is, the replica $\mathbf{v}_{p_1^{(\alpha(2))}}$ is not in $A_2$. At time $t_1^{(2)}$, the replica $\mathbf{v}_{k_{\alpha(2)}}$ sits on the stone and starts sliding through the replicas in $A_2$. At time $t_1^{(2)}+|A_2|$, the stone slides from under $\mathbf{v}_{k_{\alpha(2)}}$ to under $\mathbf{v}_{k_{\alpha(3)}}$. So, $t_1^{(3)}=t_1^{(2)}+|A_2|+1$. This concludes case two.

It follows that for $2\leq i \leq n-1$ if $\mathbf{v}_{p_1^{\alpha(i)}}$ is in $A_i$ then $t_1^{(i+1)}=t_1^{(i)}+m(N-1)+|A_i|+1$; conversely, if $\mathbf{v}_{p_1^{(\alpha(i))}}$ is not in $A_i$, then $t_1^{(i+1)}=t_1^{(i)}+|A_i|+1$. Now consider $i=n$; at time $t_1^{(n)}$, the replica $\mathbf{v}_{\alpha(0)}$ sits on the stone and starts sliding through the replicas in $A_n \cup \textbf{v}_{p_1^{(\alpha(1))}} 
\cup A_1$. If $\mathbf{v}_{p_1^{(\alpha(n))}}$ is in $A_n \cup A_1$, it follows that $t_2^{(1)}=t_1^{(n)}+m(N-1)+|A_n \cup \textbf{v}_{p_1^{(\alpha(1))}} 
\cup A_1|+1$. Conversely, if $\mathbf{v}_{p_1^{(\alpha(n))}}$ is not in $A_n \cup A_1$, then $t_2^{(1)}=t_1^{(n)}+|A_n \cup A_0|+1$. 

In order for $\mathbf{v}_{k_{\alpha(1)}}=\mathbf{v}_{k_1}$ to move $nm-1$ positions clockwise (as required), each $\mathbf{v}_{k_{\alpha(i)}}$ for $2\leq i \leq n$ must move through all $A_i$. Since each $\mathbf{v}_{p_1^{(\alpha(i))}}$ for $2\leq i \leq n$ is in some set $A_i$, we can assume without loss of generality that for $2\leq i \leq n$ the replica $\mathbf{v}_{p_1^{(\alpha(i))}}$ is in the set $A_i$. In other words, we can assume that the coin travels on each $\mathcal{T}^{(i)}$, for $2\leq i \leq n$, during lap $1$.

 Then, for $2\leq i \leq n-1$, we have that  $$t_1^{(i+1)}=t_1^{(i)}+m(N-1)+|A_i|+1.$$ Additionally, $$t_2^{(1)}=t_1^{(n)}+m(N-1)+|A_n\cup A_1|+1.$$ Recall that $t_1^{(2)}=t_1^{(1)}+|A_1|+1$. We can now solve for $t_2^{(1)}$. Observe,
\begin{align*}
   t_2^{(1)}&=t_1^{(n)}+m(N-1)+|A_n\cup A_1|+1\\
   &=t_1^{(n-1)}+m(N-1)+|A_{n-1}|+1)+m(N-1)+|A_n\cup A_1|+1\\
   &=t_1^{(1)}+ \sum_{i=1}^{n} |A_i|+|A_1|+(n-1)m(N-1)+n\\
   &=t_1^{(1)}+ (N-n)+|A_1|+(n-1)m(N-1)+n\\
   &=t_1^{(1)}+N+|A_1|+(n-1)(m(N-1)).\\
\end{align*}
We have found that $$t_2^{(1)}=t_1^{(1)}+N+|A_1|+(n-1)(m(N-1)).$$

It follows from this analysis that the clockwise cyclic order of replicas in $\SD_{t_2^{(1)}}$, starting with $\mathbf{v}_{k_{\alpha(1)}}=\mathbf{v}_{k_1}$, is given by
$$\mathbf{v}_{k_{\alpha(1)}}, A_2, \textbf{v}_{k_{\alpha(2)}}, \ldots, A_n, \textbf{v}_{k_{\alpha(n)}}, A_1.$$

We will now evaluate how the stone diagram evolves during lap $2$. At time $t_2^{(1)}$, the replica $\mathbf{v}_{k_{\alpha(1)}}$ sits on the stone and starts moving through the replicas in $A_2$. At time $t_2^{(1)}+|A_2|$ the stone slides from under $\mathbf{v}_{k_{\alpha(1)}}$ to under $\mathbf{v}_{k_{\alpha(2)}}$. So, $t_2^{(2)}=t_2^{(1)}+|A_2|+1$. In general, for $2\leq i \leq n$, we have that $t_2^{(i)}=t_2^{(i-1)}+|A_i|+1$. At time $t_2^{(n)}$, the replica $\mathbf{v}_{k_{\alpha(n)}}$ sits on the stone and slides through the replica in $A_{1} \cup A_{2}$. At time $t_2^{(n)}+|A_{1} \cup A_{2}|$ the stone slides from under $\mathbf{v}_{k_{\alpha(n)}}$ to under $\mathbf{v}_{k_{\alpha(1)}}$. Therefore $t_3^{(1)}=t_2^{(n)}+|A_{1} \cup A_{2}|+1$. We can solve for $t_3^{(1)}$:
\begin{align*}
    t_3^{(1)}&=t_2^{(3)}+|A_1\cup A_2|+1\\
    &=t_2^{(1)}+\sum_{i=1}^n{|A_i|}+|A_2|+n\\
    &=t_2^{(1)}+(N-n)+|A_2|+n\\
    &=t_2^{(1)}+N+|A_2|.
\end{align*}
Since $p_1^{(\alpha(i))}$ is in $A_i$ for $2\leq i \leq n$ and $p_1^{(\alpha(1))}$ is in $A_n$, during lap $2$ the coin doesn't travel across any $\mathcal{B}^{(i)}$. In general, during lap $\ell$ for $2\leq \ell \leq n-1$ the replica $\mathbf{v}_{k_{\alpha(i)}}$ for $1\leq i \leq n-1$ will move though the set $A_{i+\ell-1}$. Additionally, the replica $\mathbf{v}_{k_{\alpha(n)}}$ will move through $A_{n+\ell-1} \cup A_{\ell}$. It follows that for $2\leq \ell \leq n$, we have that $t_\ell^{(1)}=t_{\ell-1}^{(1)}+N+|A_{\ell-1}|$. Recall that $t_2^{(1)}=t_1^{(1)}+N+|A_1|+(n-1)(m(N-1))$. We can now solve for $t_n^{(1)}$. Observe,
\begin{align*}
t_n^{(1)}&=t_{n-1}+N+|A_{n-1}|\\
&=t_{n-2}+N+|A_{n-2}|++N+|A_{n-1}|\\
&=t_2^{(1)}+(n-2)N+\sum_{i=2}^{n-1}|A_i|\\
&=t_1^{(1)}+(n-1)N+\sum_{i=1}^{n-1}|A_i|+(n-1)(m(N-1)).\\   
\end{align*}

At time $t_n^{(1)}$, the replica $\mathbf{v}_{k_{\alpha(1)}}$ sits on the stone and starts moving through the replicas in $A_n$ until the replica $p_1^{(\alpha(1))}$ sits one position clockwise of $\mathbf{v}_{k_{\alpha(1)}}$. At time $t_n^{(1)}+|A_n\setminus p_1^{(\alpha(1))}|$ the stone slides from under $\mathbf{v}_{k_{\alpha(1)}}$ to under $p_1^{(\alpha(1))}$. Therefore $t'=t_n^{(1)}+|A_n\setminus p_1^{(\alpha(1))}|$. Recall that $t_1^{(1)}=t+m(N-1)+1$. We can now solve for $t'$. We have
\begin{align*}
    t'&=t_n^{(1)}+|A_n\setminus p_1^{(\alpha(1))}|\\
    &=t_1^{(1)}+(n-1)N+(\sum_{i=1}^{n-1}|A_i|)+(n-1)(m(N-1))+|A_n\setminus p_1^{(\alpha(1))}|\\
    &=t_1^{(1)}+(n-1)N+(\sum_{i=1}^{n}|A_i|)-1+(n-1)(m(N-1))\\
    &=t_1^{(1)}+(n-1)N+(N-n)-1+(n-1)(m(N-1))\\
    &=t+m(N-1)+1+(n-1)N+(N-n)-1+(n-1)(m(N-1))\\
    &=t+m(N-1)+(n-1)N+(N-n)+(n-1)(m(N-1))\\
    &=t+(n-1)N+(N-n)+n(m(N-1))\\
    &=t+(n-1)+(n-1)(N-1)+(N-n)+n(m(N-1))\\
    &=t+n(N-1)+nm+(N-1)\\
    &=t+N(N-1).
\end{align*}
We have now shown that $t'=t+N(N-1)$ as desired. 

We will now show that $\SD_t=\SD_{t+N(N-1)}$ by analyzing the clockwise and counterclockwise movement of each replica in the interval $[t,t']$. First, we will divide the set of replicas into three classes:
\begin{enumerate}
    \item $\mathbf{v}_{p_j^{(i)}}$ for $2\leq j \leq m$ and $1\leq i \leq n$,
    \item $ \textbf{v}_{k_i}$ for $1 \leq i \leq n$,
    \item and $\mathbf{v}_{p_1^{(i)}}$ for $1 \leq i \leq n$.
\end{enumerate}
We will start with class $1$, namely replicas of the form $\mathbf{v}_{p_j^{(i)}}$ for $ 2 \leq j \leq m$ and $ 1 \leq i \leq n$. By \Cref{LEM:TreeBridgeSimple}, each time the coin exits and enters $\mathcal{K}$, the replica $\mathbf{v}_{p_j^{(i)}}$ moves counterclockwise $m$ spaces. This happens $n$ times. So,  $\mathbf{v}_{p_j^{(i)}}$ moves $nm$ spaces counterclockwise when the coin in not in $\mathcal{K}$. It's clear from our lap-based analysis that when the coin is on $\mathcal{K}$, the replica $\mathbf{v}_{p_j^{(i)}}$ moves $n$ spaces counterclockwise. Thus, in the interval $[t,t']$ each $\mathbf{v}_{p_j^{(i)}}$ for $2\leq j \leq m$ and $1\leq i \leq n$ moves a total of $nm+n$ spaces counterclockwise. Since $nm+n=N$, each replica in class $1$ is in the same position in both $\SD_t$ and $\SD_{t'}$. 

Now let's consider class $2$, namely replicas of the form $\mathbf{v}_{k_i}$ for $1 \leq i \leq n$. Each time the coin travels from $k_i$ to $p_1^{(i)}$ and then back from $p_1^{(i)}$ to $k_i$, the replica $\mathbf{v}_{k_i}$ moves a total of $m-1$ spaces counterclockwise. We showed that this happens once in the interval $[t,t']$. All other times, when the coin exits and enters $\mathcal{K}$, the replica $\mathbf{v}_{k_i}$ moves $m$ spaces counterclockwise. We showed that in the interval $[t,t']$ this happens $n-1$ times. Therefore, when the coin is not in $\mathcal{K}$, the replica $\mathbf{v}_{k_i}$ moves a total of
\begin{align*}
    m(n-1)+m-1=mn-1
\end{align*}
spaces counterclockwise. 

It follows directly from our lap-based analysis that when the coin is in $\mathcal{K}$, the replica $\mathbf{v}_{k_i}$ moves $nm-1$ spaces clockwise. Therefore, during the interval $[t,t']$, each $\mathbf{v}_{k_i}$ moves a total of $0$ spaces. So, each replica in class $2$ is in the same position in both $\SD_{t}$ and $\SD_{t'}$. 

Finally, consider replicas in class $3$, namely replicas of the form $\mathbf{v}_{p_1^{(i)}}$ for $1 \leq i \leq n$. Each time the coin travels from $k_i$ to $p_1^{(i)}$ and then back from $p_1^{(i)}$ to $k_i$, the replica $\mathbf{v}_{p_1^{(i)}}$ moves a total of $m+1$ spaces counterclockwise. We showed that in the interval $[t,t']$ this happens once. All other times, when the coin exits and enters $\mathcal{K}$, the replica $\mathbf{v}_{p_1^{(i)}}$ moves $m$ spaces counterclockwise. We showed that in the interval $[t,t']$ this happens $n-1$ times. Hence, when the coin is not in $\mathcal{K}$ the replica $\mathbf{v}_{p_1^{(i)}}$ moves a total of
\begin{align*}
    m(n-1)+m+1=mn+1
\end{align*}
spaces counterclockwise. 

It follows directly from our lap-based analysis that when the coin is in $\mathcal{K}$, the replica $\mathbf{v}_{p_1^{(i)}}$ moves $n-1$ spaces counterclockwise. Therefore, during the interval $[t,t']$ each $\mathbf{v}_{p_1^{(i)}}$ moves a net amount of $0$ spaces. So, each replica in class $3$ is in the same position in both $\SD_{t}$ and $\SD_{t'}$.

Since all replicas are in class $1,2$ or $3$, it is shown that all replicas are in the same positions in $\SD_t$ and $\SD_{t+N(N-1)}$. Furthermore, the stone sits on $\mathbf{v}_{k_1}$ in both $\SD_t$ and $\SD_{t+N(N-1)}$. Hence, we can conclude that $\SD_t=\SD_{t+N(N-1)}$. This completes the proof of \Cref{THM:Corona}.
\end{proof}

\section{Future Work}
\label{Sec:FutureWorks}
In this section, we discuss two directions for future work. The first is describing the behavior of toric promotion on iterative bridge sums. The second is describing the behavior of toric promotion on bridge sums of families of graphs for which the orbit length of toric promotion depends on the initial labeling $\sigma_0$.

\subsection{Iterative Bridge Sums}

The orbit length of toric promotion on trees and complete graphs individually does not depend on the initial labeling. Additionally, we showed that the orbit length of a single bridge sum of any combination of trees and complete graphs does not depend on the initial labeling. We conjecture that the same is true for any finite number of bridge sums of trees and complete graphs. Since any finite number of bridge sums of trees will result in a tree, this case is shown in \Cref{THM: Tree} \cite{D2023}.

Consider the graph $G$ constructed by taking the bridge sum of $q$ graphs $$G_1=(V_1,E_1),\ldots,G_q=(V_q,E_q)$$ on $n_1,\ldots,n_q$ vertices respectively, such that the graph $G_1$ is bridge summed with the graph $G_2$, the graph $G_2$ is bridge summed with the graph $G_3$, and so on up to the graph $G_q$ at arbitrary vertices, where each $G_i$ is either a complete graph or a tree. Let the vertex set of $G$ which is $V_1 \cup \cdots \cup V_q$ be denoted $V'$, and let $\lvert V' \rvert$ be $N$.

\begin{conjecture}
     For all such $G$ the orbit length of toric promotion on $G$ is given by $N(N-1)$.
\label{CONJ:InteriveTreeandCompelete}
\end{conjecture}

Furthermore, this prompts the broader question of characterizing, for a graph $G$ with $N$ vertices, which graphs (or families of graphs) have the property that all orbits of toric promotion on $G$ have length $N(N-1)$.

\subsection{Cycles}

This manuscript describes the specific orbit length of toric promotion on bridge sums with complete graphs and trees, two families of graphs for which the orbit length of toric promotion is independent of the initial labeling. A natural next step is to build on Lemmas~\ref{LEM:TreeBridgeSimple} and \ref{LEM:CompleteBridgeSimple}, by describing the specific orbit length of toric promotion on the bridge sum of either a tree or a complete graph, with a simple graph for which the orbit length depends on the initial labeling. One such family of simple graphs is the family of cycle graphs. 

In \cite{ADS2024}, the authors describe the orbit length of toric promotion on cycle graphs. They show that it depends on a parameter $m_\sigma$, the winding number of the cycle graph educed by the labeling $\sigma_0$. We refer the reader to \cite{ADS2024} for more about toric promotion on cycle graphs and the parameter $m_\sigma$. Here, we will define an analogous parameter $w_\sigma$ on $T(v_i)\Bridge C(v_j)$ and $K(v_i)\Bridge C(v_j)$, where $C=(V_C,E_C)$ is the cycle graph on $m$ vertices, $K=(V_K,E_K)$ the complete graph on $n$ vertices, and $T=(V_T,E_T)$ is a tree on $n$ vertices. Intuitively, $w_\sigma$ is the winding number of the induced subgraph on $V_C$, which depends on the formal ordering of $V_C$ induced by $\sigma$. 

Let $t$ be a time when the coin moves from $v_j$ to $v_i$, and let $V'$ be the vertex set of $K$ or $T$, respectively. Let $\mathbf{v}_{\omega(0)}, \ldots, \mathbf{v}_{\omega(n)}$ be the replicas of the vertices in $V'$ indexed so that they appear in clockwise cyclic order in $\SD_t$ and so that $\mathbf{v}_{\omega(0)}=v_i$. Then, let $R_0$ be $1$ more than the number of replicas in the set $\{\mathbf{v}_{\omega(1)}, \ldots, \mathbf{v}_{\omega(n)}\}$ that we cross when walking from $\mathbf{v}_{\omega(n-1)}$ to $\mathbf{v}_{\omega(1)}$. For $1\leq k \leq n-2$, let $R_k$ be $1$ more than the number of replicas in the set $\{\mathbf{v}_{\omega(1)}, \ldots, \mathbf{v}_{\omega(n)}\}$ that we cross when walking from $\mathbf{v}_{\omega(k)}$ to $\mathbf{v}_{\omega(k+1)}$. Then $w_\sigma$ is given by $$\sum_{\ell=k}^{n-2} R_{\ell}=w_\sigma(m+n-1).$$

\begin{conjecture}
Let $T(v_i)\Bridge C(v_j)$ be a bridge sum of a tree on $m$ vertices and the cycle graph on $\nu$ vertices. The orbit length of toric promotion on $T(v_i)\Bridge C(v_j)$ is given by $$w_\sigma(\nu+m-1)(\nu+m).$$
\end{conjecture}

\begin{conjecture}
Let $K(v_i)\Bridge C(v_j)$ be a bridge sum of a complete graph on $n$ vertices and the cycle graph on $\nu$ vertices. The orbit length of toric promotion on $K(v_i)\Bridge C(v_j)$ is given by $$w_\sigma(n+\nu-1)(n+\nu).$$
\end{conjecture}

\section{Acknowledgments}

This research was conducted at the University of Minnesota Duluth REU with support from Jane Street Capital, NSF Grant 2409861, and donations from Ray Sidney and Eric Wepsic. I am immensely grateful to Joe Gallian and Colin Defant for providing this wonderful opportunity.

I would like to thank Eliot Hodges, Noah Kravitz, Mitchell Lee, Rupert Li, and  Maya Sankar for their inspiring suggestions and unwavering guidance throughout the research process. Additionally, I would like to thank Sean Li, David Moulton, Ilaria Seidel, and Katherine Tung for their helpful conversations. Finally, I would like to thank Elise Catania for her thoughtful suggestions and insightful feedback. 

\bibliographystyle{alpha}
\fontsize{10pt}{10pt}\selectfont
\newcommand{\etalchar}[1]{$^{#1}$}

\end{document}